\documentclass[12pt]{article}
\usepackage{latexsym}
\usepackage{amsmath, amsthm, amsfonts}
\usepackage{graphics}
\usepackage{graphicx}
\usepackage{color}
\usepackage[table,xcdraw]{xcolor}
\usepackage{float}
\usepackage{array}
\usepackage{hyperref}

\newcommand{\abs}[1]{\left\lvert#1\right\rvert}
\newtheorem{theorem}{Theorem}[section]
\newtheorem{lemma}[theorem]{Lemma}
\newtheorem{openproblem}[theorem]{Open Problem}
\newtheorem{conjecture}{Conjecture}[section]

\def\smallskip{\addvspace{\smallskipamount}}
\def\medskip{\addvspace{\medskipamount}}
\def\bigskip{\addvspace{\bigskipamount}}
\def\makefootline{\baselineskip=24pt \line{\the\footline}}
\def\pagecontents{\ifvoid\topins\else\unvbox\topins\fi
   \dimen@=\dp255 \unvbox255
   \ifvoid\footins\else
      \vskip\skip\footins \footnoterule \unvbox\footins\fi
     \ifr@ggedbottom \kern-\dimen@ \vfil \fi}
\def\footnoterule{\kern-3pt\hrule width 2truein \kern 2.6pt}

\begin{document}

\title{{Complex Dynamics of the Difference Equation $\displaystyle z_{n+1}=\frac{\alpha}{z_{n}}+ \frac{\beta}{z_{n-1}}$ }}

\author{Sk. Sarif Hassan \thanks{Corresponding author},  Pallab Basu\\
  \small {International Centre for Theoretical Sciences}\\
  \small {Tata Institute of Fundamental Research}\\
  \small {Bangalore $560012$, India}\\
  \small Email: {\texttt{sarif.hassan@icts.res.in, pallab.basu@icts.res.in}}\\
}

\maketitle
\begin{abstract}
\noindent The dynamics of the second order rational difference equation in the title with complex parameters and arbitrary complex initial conditions is investigated. Two associated difference equations are also studied. The solutions in the complex plane of such equations exhibit many rich and complicated asymptotic belabour. The analysis of the local stability of these three difference equations and periodicity have been carried out. We further exhibit several interesting characteristics of the solutions of this equation, using computations, which does not arise when we consider the same equation with positive real parameters and initial conditions. Many interesting observations led us to pose several open problems and conjectures of paramount importance regarding chaotic and higher order periodic solutions and global asymptotic convergence of such difference equations. It is our hope that these observations of these complex difference equations would certainly be  new add-ons to the present art of research in rational difference equations in understanding the behaviour in the complex domain.
\end{abstract}

\vfill
\begin{flushleft}\footnotesize
{Keywords: Difference equation, Local asymptotic stability, Chaos, Fractals and Periodicity. }
\end{flushleft}

\section{Introduction and Preliminaries}

Consider the difference equation

\begin{equation}
z_{n+1}=\frac{\alpha}{z_{n}}+ \frac{\beta}{z_{n-1}},\qquad n=0,1,\ldots
\label{equation:total-equationA}
\end{equation}%
where the parameters $\alpha, ~\beta$ are complex numbers, and the initial conditions $%
z_{-1}$ and $z_{0}$ are arbitrary complex numbers.

\addvspace{\bigskipamount}

The same difference equation is studied when the parameters $\alpha$ and $\beta$ and initial conditions are non-negative real numbers, Eq.(\ref{equation:total-equationA}) was investigated in \cite{B-J} and \cite{C-I} where the global asymptotic stability of the positive equilibrium was proved for all the parameters. In this present article it is an attempt to understand the same in the complex plane.

\addvspace{\bigskipamount}

The set of initial conditions ${z_{-1}, ~z_{0}} \in \mathbb{C}$ for which the solution of Eq.(\ref{equation:total-equationA}) is well defined for all $n \geq 0$ is called the $\emph{good}$ set of initial conditions or the \emph{domain of definition}. It is the compliment of the $\emph{forbidden}$ set of Eq.(\ref{equation:total-equationA}) for which the solution is not well defined for some $n \geq 0$. See \cite{G-K-S} for the definition and further work on obtaining the forbidden set for the first order Riccati Difference equation. However, for the second and higher order rational difference equations, the lack of an explicit form for the solutions, makes it very challenging to obtain the good set. For Eq.(\ref{equation:total-equationA}) we pose the following open problem of paramount importance and difficulty.

\addvspace{\bigskipamount}
\begin{openproblem}
\label{openproblem:good set} Determine the $\emph{good}$  set of initial conditions for Eq.(\ref{equation:total-equationA}).

\end{openproblem}

\addvspace{\bigskipamount}
Our goal is to investigate the character of the solutions of Eq.(\ref{equation:total-equationA}) when the parameters are real or complex and the initial conditions are arbitrary complex numbers in the domain of definition. For the rest of the sequel we assume that the initial conditions are from the $\emph{good}$ set \cite{S-E}.

We now present some preliminary material which will be useful in our
investigation of Eq.(\ref{equation:total-equationA}).

\bigskip

Let $f:\mathbb{C}^{\mathcal{K}+1} \rightarrow \mathbb{C}$ be a continuous function,
where $\mathcal{K}$ is a non-negative integer and $\mathbb{C}$ is an
interval of complex numbers. Consider the difference equation
\begin{equation}
\label{equation:introduction}
z_{n+1} = f(z_{n}, z_{n-1}, \ldots, z_{n-\mathcal{K}}) \hspace{.25in}
, \hspace{.25in} n=0,1,\ldots
\end{equation}
with initial conditions $z_{-\mathcal{K}}, z_{-\mathcal{K}+1},
\ldots, z_{0} \in \mathbb{C}.$

\bigskip

We say that $\bar{y}$ is an {\it equilibrium point} of
Eq.(\ref{equation:total-equationA}) if
$$f(\bar{z}, \bar{z}, \ldots, \bar{z}) = \bar{z}.$$

\bigskip

We now impose the further restriction that the function $f(u_{0},
u_{1}, \ldots, u_{\mathcal{K}})$ be \mbox{continuously}
differentiable.

\bigskip

The \emph{linearized equation} of Eq.(\ref{equation:introduction})
about the equilibrium $\bar{z}$ is the linear difference equation

\begin{equation}
\label{equation:linearized-equation}
\displaystyle{
z_{n+1} = a_{0} z_{n} + a_{1}z_{n-1} + \cdots + a_{\mathcal{K}}
z_{n-\mathcal{K}} \hspace{.25in} , \hspace{.25in} n=0,1,\ldots
}
\end{equation}
where for each $i = 0, 1, \ldots, \mathcal{K}$
$$a_{i} = \frac{\partial f}{\partial u_{i}}(\bar{y}, \bar{y}, \ldots,
\bar{y}).$$
The \emph{characteristic equation} of
Eq.(\ref{equation:linearized-equation}) is the equation

\begin{equation}
\label{equation:characteristic-roots}
\lambda^{\mathcal{K}+1} - a_{0}\lambda^{\mathcal{K}} -
a_{1}\lambda^{\mathcal{K}-1} - \cdots - a_{\mathcal{K} - 1} \lambda -
a_{\mathcal{K}} = 0.
\end{equation}

\bigskip

The following result, called the Linearized Stability Theorem, is
useful in determining the local stability character of the equilibrium
$\bar{z}$ of Eq.(\ref{equation:introduction}), \cite{S-H}.

\bigskip

\noindent {\bf Theorem A (The Linearized Stability Theorem)} \newline
\emph{The following statements are true:}
\begin{enumerate}

\item \emph{
If every root of Eq.(\ref{equation:characteristic-roots}) has modulus less than one, then the equilibrium $\bar{z}$ of
Eq.(\ref{equation:introduction}) is locally asymptotically stable.
}

\item \emph{
If at least one of the roots of
Eq.(\ref{equation:characteristic-roots}) has modulus greater than one, then the equilibrium $\bar{z}$ of Eq.(\ref{equation:introduction}) is unstable.
}

\end{enumerate}

\bigskip

The equilibrium solution $\bar{z}$ of Eq.(\ref{equation:introduction}) is called \emph{hyperbolic} if no
root of Eq.(\ref{equation:characteristic-roots}) has modulus equal to one. If there exists a root of
Eq.(\ref{equation:characteristic-roots}) with modulus equal to one, then $\bar{z}$ is called \emph{non-hyperbolic}.

\bigskip

The equilibrium point $\bar{z}$ of Eq.(\ref{equation:introduction}) is called a \emph{sink} if every root
of Eq.(\ref{equation:characteristic-roots}) has modulus greater than one.

\bigskip

The equilibrium point $\bar{z}$ of Eq.(\ref{equation:introduction}) is called a \emph{saddle point
equilibrium} if it is hyperbolic, and if in addition, there exists a root of Eq.(\ref{equation:characteristic-roots}) with modulus
less than one and another root of
Eq.(\ref{equation:characteristic-roots}) with modulus greater
than one. In particular, if $\bar{z}$ is a saddle point equilibrium of
Eq.(\ref{equation:introduction}), then $\bar{z}$ is unstable.

\bigskip

The equilibrium point $\bar{z}$ of
Eq.(\ref{equation:introduction}) is called a \emph{repeller} if every root
of Eq.(\ref{equation:characteristic-roots}) has absolute value less
than one.

\addvspace{\bigskipamount}

The following theorems would be useful in determining the characteristics of zeros of the \emph{characteristic polynomial}.\\ \\

\noindent {\bf Theorem B (Enestrom-Kakeya Theorem)} \newline
\\
Consider a polynomial $P(z)=\sum_{j=0}^n a_j z^j$ of degree $n$ with real coefficients $a_0, a_1, \dots a_n$ such that $a_n \geq a_{n-1} \geq a_{n-1} \dots \geq a_0 > 0$, then $P(z)$ has all its zeros in for $\mid z\mid \leq 1$ \cite{J-L}. In literature \cite{P-V}, there are some extensions and generalizations of this theorem. Interestingly, in 1967, Govil and Rahaman \cite{G-R} extended it to a polynomial with complex coefficients by proving the following theorem.
\\ \\
\noindent {\bf Theorem C (Govil-Rahaman Theorem)} \newline
\\
Let $P(z)=\sum_{j=0}^n a_j z^j$ be a polynomial of degree $n$ with complex coefficients $a_0, a_1, \dots a_n$, there exists an $a > 0$ such that $$ \mid a_n \mid  \geq  a \mid a_{n-1} \mid \geq a^2 \mid a_{n-1} \mid \geq \dots \geq a^{n-1} \mid a_1 \mid \geq a^{n} \mid a_0 \mid$$ then the polynomial $P(z)$ has all its zeros in $\mid z \mid \leq \frac{1}{a} R_1$ where $R_1$ is the greatest positive root of the cubic trinomial $R^{n+1}-2R^n+1=0$. \\ \\

\section{Local Stability of the Equilibriums}

\label{section:positive-equilibrium} In this section we establish the local stability character of the equilibria of Eq.(\ref{equation:total-equationA}) when the parameters $\alpha$ and $\beta$ are considered to be complex numbers with the initial conditions are arbitrary complex numbers.

\addvspace{\bigskipamount}

\noindent The equilibrium points of Eq.(\ref{equation:total-equationA}) are
the solutions of the equation
\[
\bar{z}=\frac{\alpha}{\bar{z}}+ \frac{\beta}{\bar{z}}
\]
Eq.(\ref{equation:total-equationA}) has two equilibria points
$$ \bar{z}_{1,2} = \pm \sqrt{\alpha+\beta}.$$
The linearized equation of Eq.(\ref{equation:total-equationA}) with respect to the equilibrium $\bar{z}_{1, 2}$ is
\[
Z_{n+1} + \frac{ \alpha }{ \alpha+ \beta }Z_{n} + \frac{\beta}{ \alpha + \beta}Z_{n-1} = 0, \qquad n
= 0,1,\ldots,
\]
with associated characteristic equation
\begin{equation}
\lambda^{2} + \frac{\alpha}{\alpha+ \beta} \lambda +\frac{\beta}{\alpha+ \beta} = 0.
\end{equation}

\addvspace{\bigskipamount}

The following result gives the local asymptotic stability of the equilibrium $\bar{z}_{1,2}$.

\begin{lemma}
\label{Result:positive-local-stability1} The equilibriums $\bar{z}_{1,2}$ $= \pm \sqrt{\alpha+\beta}$ of Eq.(\ref{equation:total-equationA}) are locally asymptotically stable if and only if $a\abs{\frac{\beta}{\alpha+ \beta}} \leq \abs{\frac{\alpha}{\alpha+ \beta}} \leq \frac{1}{a} < \frac{1}{\phi}$ for some $a > 0$, $\phi$ denotes \emph{golden ratio}.
\end{lemma}

\addvspace{\bigskipamount}

\begin{proof}
From \textbf{Theorem C}, the all roots of the characteristic equation $(5)$
lie in the close ball $\abs{z} \leq \frac{1}{a} K_1$ where $K_1$ is the largest zeros of the cubic polynomial $K^3-2K^2+1$ such that for some $a>0$
\begin{equation}
 1 \geq a \abs{\frac{\alpha}{\alpha+ \beta}} \geq a^2 \abs{\frac{\beta}{\alpha+\beta}}
 \end{equation}
is satisfied.

In this case $$K_1=\frac{\sqrt{5}+1}{2}$$ which is known also as \emph{golden ratio}, $\phi$. Therefore the closed ball becomes
\begin{equation}
\abs{z} \leq \frac{\phi}{a}
\end{equation}
From \textbf{Theorem (A)} we have if every roots of the Eq. (5) has modulus less than one, then the equilibriums $\bar{z}_{1,2}$ are \emph{locally asymptotically stable}. Then from Eq. $(7)$, the radius of the ball $\abs{z} \leq \frac{\phi}{a}$ should be $\frac{\phi}{a} < 1$. In turns, it is $a > \phi$.\\
\noindent

Therefore, The equilibriums of the Eq. $(6)$  $\bar{z}_{1,2}$ $= \pm \sqrt{\alpha+\beta}$ of Eq.(\ref{equation:total-equationA}) are locally asymptotically stable if and only if $a\abs{\frac{\beta}{\alpha+ \beta}} \leq \abs{\frac{\alpha}{\alpha+ \beta}} \leq \frac{1}{a} < \frac{1}{\phi}$.\\
\noindent

\end{proof}

\begin{lemma}
\label{Result:positive-local-stability2} The equilibriums $\bar{z}_{1,2}$ $= \pm \sqrt{\alpha+\beta}$ of Eq.(\ref{equation:total-equationA}) are \emph{sink} if and only if $a\abs{\frac{\beta}{\alpha+ \beta}} \leq \abs{\frac{\alpha}{\alpha+ \beta}} \leq \frac{1}{a}$ for some $0<a<\phi$, $\phi$ denotes \emph{golden ratio}.
\end{lemma}

\begin{proof}
From the \textbf{Theorem (A)}, we have if every roots of the Eq. (5) has modulus greater than one, then the equilibriums $\bar{z}_{1,2}$ are \emph{sink}.
Then from Eq. $(7)$, the radius of the ball $\abs{z} \leq \frac{\phi}{a}$ should be $\frac{\phi}{a} > 1$. In turns, it is $a < \phi$.\\
\noindent
Therefore, The equilibriums $\bar{z}_{1,2}$ $= \pm \sqrt{\alpha+\beta}$ of Eq.(\ref{equation:total-equationA}) are sink if and only if $a\abs{\frac{\beta}{\alpha+ \beta}} \leq \abs{\frac{\alpha}{\alpha+ \beta}} \leq \frac{1}{a}$ for some $0<a<\phi$.

\end{proof}

\begin{lemma}
\label{Result:positive-local-stability3} The equilibriums $\bar{z}_{1,2}$ $= \pm \sqrt{\alpha+\beta}$ of Eq.(\ref{equation:total-equationA}) are \emph{hyperbolic} if and only if $\abs{\frac{\frac{-\alpha}{\alpha+\beta}\pm \sqrt{(\frac{-\alpha}{\alpha+\beta})^2-4\frac{\beta}{\alpha+\beta}}}{2}} \neq 1 $.
\end{lemma}

\begin{proof}
 The two roots $\lambda_{\pm}$ (say) of the Eq. $(6)$ are $$ \lambda_{\pm}(\alpha, \beta)=\frac{\frac{-\alpha}{\alpha+\beta}\pm \sqrt{(\frac{-\alpha}{\alpha+\beta})^2-4\frac{\beta}{\alpha+\beta}}}{2} $$ From the \textbf{Theorem (A)},the equilibriums $\bar{z}_{1,2}$ are \emph{hyperbolic} if the modulus of the roots of the Eq. $(6)$ are not equal to $1$. If there exists a root of Eq. $(6)$ with modulus equal to $1$S, then the equilibrium is called \emph{non-hyperbolic}. So, if the modulus of $\lambda_{\pm}$ should be non equal to 1 then the equilibriums $\bar{z}_{1,2}$ are \emph{hyperbolic}. That is, $\abs{\lambda_{\pm}} \neq 1$. \\ In addition to the hyperbolicity of the equilibriums, if $\abs{\lambda_{+}} > 1$ and $\abs{\lambda_{-}} < 1$ then the equilibriums $\bar{z}_{1,2}$ are \emph{saddle}.\\

If any one of the parameters $\alpha$ or $\beta = 0$, then $\abs{\lambda_{+}}=1$ then the equilibriums $\bar{z}_{1,2}$ are \emph{non-hyperbolic}.

\end{proof}

\noindent
In different specific cases of the parameters $\alpha$ and $\beta$ of the Eq. $(1)$, the stability of the equilibriums $\sqrt{\alpha+\beta}$ are characterized through the characteristic equation Eq. $(5)$ about the equilibriums of the Eq. $(1)$ are described in the Table [1],

\begin{table}[h]
\small
\centering
\begin{tabular}{c c c c}
\hline
\centering   \textbf{Parameters}  &   \textbf{Characteristic Equation} &   \textbf{Modulus of the Zeros} &   \textbf{Inference}  \\
\hline\\
\centering $\alpha=\beta$ & $\lambda^2+\frac{1}{2}\lambda+\frac{1}{2}=0$ & $\mid\lambda_{\pm}\mid>1$ & $\pm \sqrt{2\alpha}$ are \emph{sinks}  \\\\
\hline\\
\centering $\alpha=i\beta$ & $\lambda^2+\frac{i}{1+i}\lambda+\frac{1}{1+i}=0$ &  $\mid\lambda_{\pm}\mid<1$ & $\pm\sqrt{\beta+i\beta}$ are \emph{l.a.s} \\\\
\hline\\
\centering $\alpha=-i\beta$ & $\lambda^2-\frac{i}{1-i}\lambda+\frac{1}{1-i}=0$ &  $\mid\lambda_{\pm}\mid>1$ & $\pm \sqrt{2\alpha}$ are \emph{sinks}  \\\\
\hline\\
\centering $\alpha=0$ \& $\beta \neq 0$ & $\lambda^2+1=0$ & $\mid\lambda_{\pm}\mid=1$ & $\pm \sqrt{\beta}$ are \emph{non-hyp} \\\\
\hline\\
\centering $\alpha \neq 0$ \& $\beta=0$ & $\lambda^2+\lambda=0$ &  $\mid\lambda_{\pm}\mid=0,1$ & $\pm \sqrt{\alpha}$ are \emph{non-hyp} \\\\
\hline\\

\end{tabular}
\caption{Local stability of the equilibriums of Eq. $(1)$}
\label{Table:}
\end{table}

\section{Dynamics of Associated Difference Equations}
Here we consider another two associated difference equations which are defined as follows
 \begin{equation}
 z_{n+1}=\alpha+\frac{z_n}{z_{n-1}}\beta \qquad n=0,1,\ldots
 \end{equation}

 \begin{equation}
 z_{n+1}=\frac{z_{n-1}}{z_n}\alpha+\beta \qquad n=0,1,\ldots
\end{equation}
\noindent
where the parameters $\alpha, ~\beta$ are complex numbers, and the initial conditions $z_{-1}$ and $z_{0}$ are arbitrary complex numbers.\\
\noindent
These two difference equations Eq. $(8)$ and Eq. $(9)$ are derived from Eq. $(1)$ just by multiplication of $z_n$ and $z_{n-1}$ with the function argument respectively .

\noindent
It is natural to pose a similar open problem as stated before for the Eq. $(1)$.

\begin{openproblem}
Determine the $\emph{good}$  set of initial conditions for Eq.$(8)$ and Eq. $(9)$.
\end{openproblem}
\noindent
Let us now investigate the local stability of the equilibriums of these two difference equation in the following section.

\subsection{Local Stability of the Equilibriums}
In this section we describe the local stability character of the equilibria of Eq.$(8)$ and Eq.$(9)$ when the parameters $\alpha$ and $\beta$ are considered to be complex numbers with the initial conditions are arbitrary complex numbers.
\\
The equilibrium are for the difference equations Eq.$(8)$ and Eq.$(9)$ is $\alpha+\beta$. Let us linearize the Eq.(8) and Eq.(9) about the equilibrium $\alpha+\beta$ as follows. The characteristic equations associated to the linear equations corresponding to Eq.$(8)$ and Eq.$(9)$ are
\begin{equation}
\lambda^{2} - \frac{\beta}{\alpha+ \beta} \lambda +\frac{\beta}{\alpha+ \beta} = 0.
\end{equation}

\begin{equation}
\lambda^{2} + \frac{\alpha}{\alpha+ \beta} \lambda -\frac{\alpha}{\alpha+ \beta} = 0.
\end{equation}

\noindent
In different specific cases of the parameters $\alpha$ and $\beta$ of the Eq. $(8)$, the stability of the equilibriums $\alpha+\beta$ are characterized through the characteristic equation Eq. $(10)$ about the equilibriums of the Eq. $(8)$ are described in the Table [2].

\begin{table}[H]
\small
\centering
\begin{tabular}{c c c c}
\hline
\centering   \textbf{Parameters}  &   \textbf{Characteristic Equation} &   \textbf{Modulus of the Zeros} &   \textbf{Inference}  \\
\hline\\
\centering $\alpha=\beta$ & $\lambda^2-\frac{1}{2}\lambda+\frac{1}{2}=0$ & $\mid\lambda_{\pm}\mid<1$ & $2\alpha$ is \emph{l.a.s}  \\\\
\hline\\
\centering $\alpha=i\beta$ & $\lambda^2-\frac{1}{1+i}\lambda+\frac{1}{1+i}=0$ &  $\mid\lambda_{\pm}\mid>1$ & $\beta+i\beta$ is \emph{sinks} \\\\
\hline\\
\centering $\alpha=-i\beta$ & $\lambda^2-\frac{1}{1-i}\lambda+\frac{1}{1-i}=0$ &  $\mid\lambda_{\pm}\mid>1$ & $\beta-i\beta$ is \emph{sinks}  \\\\
\hline\\
\centering $\alpha=0$ \& $\beta \neq 0$ & $\lambda^2-\lambda+1=0$ & $\mid\lambda_{\pm}\mid=1$ & $\beta$ are \emph{non-hyp} \\\\
\hline\\
\centering $\alpha \neq 0$ \& $\beta=0$ & $\lambda^2=0$ &  $\mid\lambda_{\pm}\mid=0$ & $\alpha$ are \emph{repeller} \\\\
\hline\\

\end{tabular}
\caption{Local stability of the equilibriums of Eq. $(8)$}
\label{Table:}
\end{table}

\noindent
In different specific cases of the parameters $\alpha$ and $\beta$ of the Eq. $(9)$, the stability of the equilibriums $\alpha+\beta$ are characterized through the characteristic equation Eq. $(11)$ about the equilibriums of the Eq. $(9)$ are described in the Table [3].

\begin{table}[H]
\small
\centering
\begin{tabular}{c c c c}
\hline
\centering   \textbf{Parameters}  &   \textbf{Characteristic Equation} &   \textbf{Modulus of the Zeros} &   \textbf{Inference}  \\
\hline\\
\centering $\alpha=\beta$ & $\lambda^2+\frac{1}{2}\lambda-\frac{1}{2}=0$ & $\mid\lambda_{+}\mid=1$, $\mid\lambda_{-}\mid=\frac{1}{2}$ & $2\alpha$ is \emph{non-hyp}  \\\\
\hline\\
\centering $\alpha=i\beta$ & $\lambda^2+\frac{i}{1+i}\lambda-\frac{i}{1+i}=0$ &  $\mid\lambda_{\pm}\mid<1$ & $\beta+i\beta$ is \emph{l.a.s} \\\\
\hline\\
\centering $\alpha=-i\beta$ & $\lambda^2-\frac{i}{1-i}\lambda+\frac{i}{1-i}=0$ &  $\mid\lambda_{\pm}\mid=0$ & $\beta-i\beta$ is \emph{repeller}  \\\\
\hline\\
\centering $\alpha=0$ \& $\beta \neq 0$ & $\lambda^2=0$ & $\mid\lambda_{\pm}\mid=0$ & $\beta$ are \emph{repeller} \\\\
\hline\\
\centering $\alpha \neq 0$ \& $\beta=0$ & $\lambda^2+\lambda-1=0$ &  $\mid\lambda_{\pm}\mid>1$ & $\alpha$ are \emph{sinks} \\\\
\hline\\

\end{tabular}
\caption{Local stability of the equilibriums of Eq. $(9)$}
\label{Table:}
\end{table}


\section{Periodic of Solutions}

\label{section:periodicity}

In this section we discuss the global periodicity and the existence of solutions that converge to periodic solutions of Eq.$(1)$, Eq.$(8)$ and Eq.$(9)$.\\

A difference equation is said to be \emph{globally periodic} of period $t$ if $x_{n+t}=x_n$ for any given initial condition. The characterization of periodic rational difference equations (either with real or complex coefficients) of order k is a challenging area of current research. See \cite{Ca-L} and \cite{K-L} for several periodicity results of second and third order rational difference equations, respectively, with nonnegative real coefficients and arbitrary nonnegative real initial conditions.

\subsection{Periodic Solutions of period 2 and 4}
For the difference equation Eq.$(1)$, it is our firm conviction that for $\alpha=0$ or $\beta=0$, the all solutions are periodic and of period $4$ or $2$ respectively. Also, it is conjectured that all solutions of the difference equation Eq.$(1)$ with the parameter $\beta=0$ or $\alpha=0$ are convergent and convergent to a periodic point $\sqrt{2\alpha}$ or $\sqrt{2\beta}$ of period $2$ or $4$ respectively. \\

\noindent

Further we took an attempt to investigate the higher order periodic cycles. We found there there is no periodic solutions of period $3$ which is formally proved in the next section. Interestingly there are periodic solutions of higher periods viz. $2$,$4$,$5$,$6$,$7$,\dots.

\subsection{Higher Order Cycles in Solutions}

In investigating the solutions of higher order cycles of order $d$ of the Eq. $(1)$ the following system of nonlinear equations needs to be solved. The solutions are indeed the cycles.

\begin{equation}
z_{mod(n,d)}=\frac{\alpha}{z_{mod(n-1,d)}}+\frac{\beta}{z_{mod(n-2,d)}}
\end{equation}

\noindent
for $n=1,2,3,\dots,d$ and where $mod(n,d)$ denotes the $x$ such that $n \equiv x$ $mod(d)$.
\\ \\
\noindent
\textbf{Theorem D}: There does not exist any periodic solution of period $3$.\\
\noindent
Proof: The solutions of the equations $$\left\{z_0=\frac{\alpha }{z_2}+\frac{\beta }{z_1}, z_1=\frac{\alpha }{z_0}+\frac{\beta }{z_2},z_2=\frac{\alpha }{z_1}+\frac{\beta }{z_0}\right\}$$ are the periodic solutions of period $3$. Interestingly there is no solutions except the equilibrium $\sqrt{\alpha+\beta}$ of the equation Eq. $(1)$. Hence there is no solution of period $3$.
\\\\
\noindent
A list of difference periodic cycles of different length are adumbrated for the parameters $\alpha=\beta=1$ in the Table $4$.

\begin{table}[H]
\small
\begin{tabular}{| m{2.5cm} | m{1cm}| m{12.3cm} |}
\hline
\centering   \textbf{Parameters}  &   \textbf{CL} &   \textbf{One of the Solutions} \\
\hline
\centering $\alpha=1; \beta=1$ & $4$ & $0.765367 i,-1.84776 i,-0.765367 i,1.84776 i$   \\
\hline
\centering $\alpha=1; \beta=1$ & $5$ & $0.309721 i,-1.83083 i,-2.68251 i,0.918986 i,-0.71537 i$   \\
\hline
\centering $\alpha=1; \beta=1$ & $6$ & $0.53713 i,-0.735107 i,-0.501402 i,3.35475 i,1.69632 i,-0.887595 i$ \\
\hline
\centering $\alpha=1; \beta=1$ & $7$ & $0.563218 i,-1.45984 i,-1.09051 i,1.60201 i, 0.292791 i,-4.03963 i,-3.16786 i$   \\

\hline
\end{tabular}
\caption{Cycle solutions of different order for different choice of $\alpha$ and $\beta$ of Eq. $(1)$, Here CL denotes cycle length.}
\label{Table:}
\end{table}

\noindent
There are also many higher order cycles too. For $\alpha=(15, -88)$ and $\beta=(-53,-30)$ and the initial values $z_0=(65, -97)$ and $z_1=(-92,-67)$, we have encountered a solution of the difference equation Eq. ($1$) which converges to a periodic point of period $23$. The periodic cycles is $(18.574,-4.5796) \rightarrow (3.295, -1.1914)\rightarrow(-9.4511,	-17.474)\rightarrow(15.164, -23.792)\rightarrow(3.424, 0.61574)\rightarrow(-13.604, -7.0189)\rightarrow(3.742,	-25.505)\rightarrow(2.618, 3.3548)\rightarrow(-9.7582, 5.5619)\rightarrow(-11.358, -10.844)\rightarrow(-1.2796, 5.1965)\rightarrow(0.10378, 15.669)\rightarrow(-18.574, 4.5796)\rightarrow(-3.295, 1.1914)\rightarrow(9.4511, 17.474)\rightarrow(-15.164, 23.792)\rightarrow(-3.424, -0.61574)\rightarrow(13.604, 7.0189)\rightarrow(-3.742, 25.505)\rightarrow(-2.618, -3.3548)\rightarrow(9.7582, -5.5619)\rightarrow(11.358, 10.844)\rightarrow(1.2796, -5.1965)\rightarrow(-0.10378, -15.669)\rightarrow(18.574,-4.5796)$.
\noindent
The plot of the real and imaginary sequence is given the Fig. $(1)$.

\begin{figure}[H]
      \centering

      \resizebox{12cm}{!}
      {
      \begin{tabular}{c}
      \includegraphics [scale=2.0]{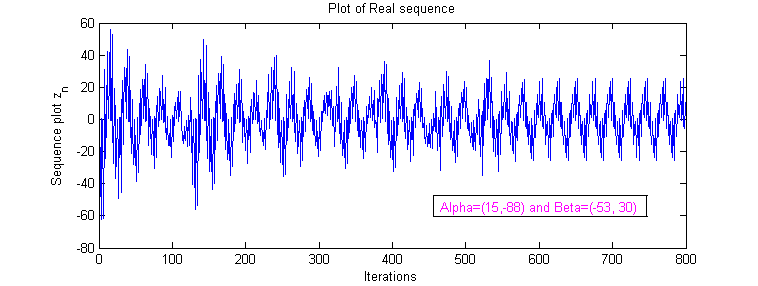}\\\\
      \includegraphics [scale=2.0]{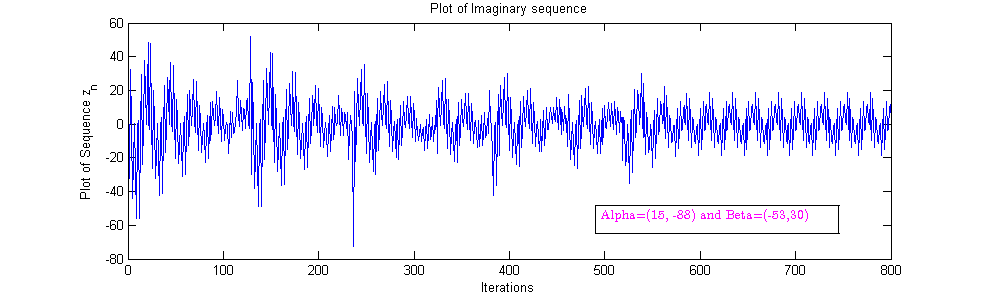}\\

      \end{tabular}
      }
\caption{Plot of Real and Imaginary Sequence for $\alpha=(15, -88)$ and $\beta=(-53,-30)$ and the initial values $z_0=(65, -97)$ and $z_1=(-92,-67)$.}
      \begin{center}

      \end{center}
      \end{figure}
\noindent
We also have listed different periodic cycles of the different length for the associated difference equations Eq. $(8)$ and Eq. $(9)$ in the Table $5$ and Table $6$.

\begin{table}[H]
\small
\begin{tabular}{| m{2.5cm} | m{1cm}| m{12cm} |}
\hline
\centering   \textbf{Parameters}  &   \textbf{CL} &   \textbf{One of the Solutions} \\
\hline
\centering $\alpha=1; \beta=1$ & $3$ & $1.24698, -1.80194, -0.445042$   \\
\hline
\centering $\alpha=1; \beta=1$ & $5$ & $0.83083,-1.91899,-1.30972,1.68251,-0.28463$   \\
\hline
\centering $\alpha=1; \beta=1$ & $6$ & $1.80194,-0.445042,1.24698,-1.80194,-0.445042,1.24698$   \\
\hline
\centering $\alpha=1; \beta=1$ & $7$ & $1.87278,-0.556474,0.702862,-0.263063,0.625725,-1.37861,-1.20322$ \\

\hline
\end{tabular}
\caption{Cycle solutions of different order for different choice of $\alpha$ and $\beta$ of Eq. $(8)$.}
\label{Table:}
\end{table}

\begin{table}[H]
\small
\begin{tabular}{| m{2.5cm} | m{1cm}| m{12cm} |}
\hline
\centering   \textbf{Parameters}  &   \textbf{CL} &   \textbf{One of the Solutions} \\
\hline
\centering $\alpha=1; \beta=1$ & $3$ & $1.24698,-0.445042,-1.80194$   \\
\hline
\centering $\alpha=1; \beta=1$ & $4$ & $1.53339-0.608009 i,1.81536+0.929423 i,1.53339-0.608009 i,1.81536+0.929423 i$   \\
\hline
\centering $\alpha=1; \beta=1$ & $5$ & $0.574313+0.798528 i,-0.273032-0.160806 i,-1.84063-1.25163 i,1.14206-0.00923437 i,-0.60271-1.1089 i$   \\
\hline
\centering $\alpha=1; \beta=1$ & $6$ & $2.61506,1.61917,2.61506,1.61917,2.61506,1.61917$ \\
\hline
\centering $\alpha=1; \beta=1$ & $7$ & $0.962688+0.453798 i,-0.251383-0.177869 i,-2.40312+0.602711 i,1.08095+0.0943187 i,-1.15807+0.745878 i,0.377343-0.482479 i,-1.12397-0.739107 i$
\\
\hline
\centering $\alpha=1; \beta=1$ & $8$ & $27.0466,1.03839,27.0466,1.03839,27.0466,1.03839,27.0466,1.03839$
\\

\hline
\end{tabular}
\caption{Cycle solutions of different order for different choice of $\alpha$ and $\beta$ of Eq. $(8)$.}
\label{Table:}
\end{table}

\noindent
\textbf{Theorem E}: There does not exist any periodic solution of period $4$.\\
\noindent
Proof: The solutions of the equations $$\left\{z_1=\alpha +\frac{\beta  z_0}{z_3},z_2=\alpha +\frac{\beta  z_1}{z_0},z_3=\alpha +\frac{\beta  z_2}{z_1},z_0=\alpha +\frac{\beta  z_3}{z_2}\right\}$$ are the periodic solutions of period $4$. There is no solutions except the equilibrium $\alpha+\beta$ of the equation Eq. $(8)$. Hence there is no solution of period $4$.\\ \\

\noindent
\textbf{Theorem F}: The 2-cycle of the difference equation Eq. $(1)$ are \emph{attracting}, \emph{repelling} and \emph{non-hyperbolic} if and only if the modulus of the eigenvalues of the matrix $$\left(
\begin{array}{cc}
 \frac{\alpha }{\beta -\alpha } & \frac{\beta }{\beta -\alpha } \\
 \frac{-\alpha ^2+\beta  \alpha +\beta ^2}{(\alpha -\beta )^2} & \frac{\alpha  \beta }{(\alpha -\beta )^2} \\
\end{array}
\right)$$ are lesser, greater than or equal to $1$ respectively.\\ \\
\noindent
Proof: The two cycles of the difference equation Eq. $(1)$ are the solutions of the nonlinear equations $$\left\{z_1=\frac{\alpha }{z_0}+\frac{\beta }{z_1},z_0=\frac{\alpha }{z_1}+\frac{\beta }{z_0}\right\}$$ Therefore the 2 cycles are $-\sqrt{\alpha-\beta}$ and $\sqrt{\alpha+\beta}$ which is actually fixed point of the Eq. $(1)$. So we are mainly interested in the other 2 cycle which is $-\sqrt{\alpha-\beta}$. \\

The characteristic equation of the linearized equation about the 2-cycle $-\sqrt{\alpha-\beta}$ of Eq. (1) is $$\left|
\begin{array}{cc}
 \frac{\alpha }{\beta -\alpha } & \frac{\beta }{\beta -\alpha } \\
 \frac{-\alpha ^2+\beta  \alpha +\beta ^2}{(\alpha -\beta )^2} & \frac{\alpha  \beta }{(\alpha -\beta )^2} \\
\end{array}\right|=0$$

The zeros of this characteristic equation, which are essentially the eigenvalues of the matrix $\left(
\begin{array}{cc}
 \frac{\alpha }{\beta -\alpha } & \frac{\beta }{\beta -\alpha } \\
 \frac{-\alpha ^2+\beta  \alpha +\beta ^2}{(\alpha -\beta )^2} & \frac{\alpha  \beta }{(\alpha -\beta )^2} \\
\end{array}
\right)$ are $$\left\{\frac{-\alpha ^2-\sqrt{\alpha ^4+4 \alpha ^3 \beta -8 \alpha ^2 \beta ^2+4 \beta ^4}+2 \alpha  \beta }{2 (\alpha -\beta )^2},\frac{-\alpha ^2+\sqrt{\alpha ^4+4 \alpha ^3 \beta -8 \alpha ^2 \beta ^2+4 \beta ^4}+2 \alpha  \beta }{2 (\alpha -\beta )^2}\right\}$$ The modulus of these two eigenvalues need to be lesser, greater or equal to $1$ in determining the $2$-cycle as \emph{attracting}, \emph{repelling} or \emph{non-hyperbolic}. Hence the theorem is proved.\\

\noindent
In the similar fashion, the local stability of the $2$-cycle of the difference equations Eq. $(8)$ and Eq. $(9)$ can also be analyzed easily which we left to reader. This result would encourage us to draw the following conjecture.
\begin{conjecture}
For every finite length (if exists) cycle of the difference equations Eq. $(1)$, Eq. $(8)$ and Eq. $(9)$, there exists at least one solution which converges to the cycle.
\end{conjecture}

\section{Chaotic Solutions}
For the difference equations Eq. $(1)$, Eq. $(8)$ and Eq. $(9)$ we have found chaotic solutions. The chaotic property of the solutions can be ensured through the largest positive Lyapunav exponent of the solutions. Few examples are illustrated in the following table $(7)$ with corresponding plots of the solutions as shown in Fig. $(2)$. The green and blue plot are denoting real and imaginary sequences respectively.

\begin{table}[H]
\small
\begin{tabular}{| m{5cm} | m{5.3cm}| m{5.2cm} |}
\hline
\centering   $z_{n+1}=\frac{\alpha}{z_{n}}+\frac{\beta}{z_{n-1}}$ &   $z_{n+1}=\alpha+\frac{z_n}{z_{n-1}}\beta$ & $z_{n+1}=\frac{z_{n-1}}{z_n}\alpha+\beta$  \\
\hline
\centering $\alpha=(30,47)$, $\beta=(30,-10)$, $z_0=(9,-41)$, $z_1=(49. -63)$ & $\alpha=(56, -22)$, $\beta=(-52, -19)$, $z_0=(-81, -74)$, $z_1=(89, 92)$  & $\alpha=(4, -81)$, $\beta=(64, 64)$, $z_0=(45, -70)$, $z_1=(32, 4)$    \\
\hline

\end{tabular}
\caption{Cycle solutions of different order for different choice of $\alpha$ and $\beta$ of Eq. $(8)$ and initial values.}
\label{Table:}
\end{table}

\begin{figure}[H]
      \centering

      \resizebox{16cm}{!}
      {
      \begin{tabular}{c}
      \includegraphics [scale=3]{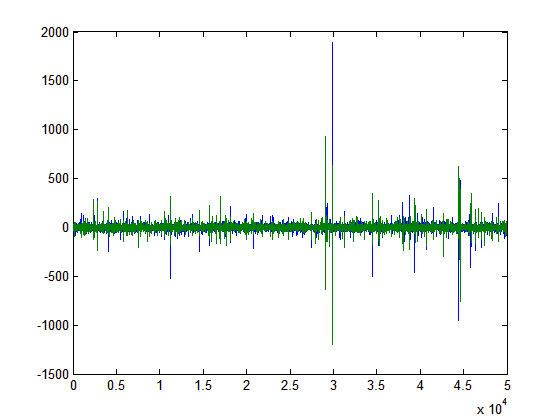}
      \includegraphics [scale=3]{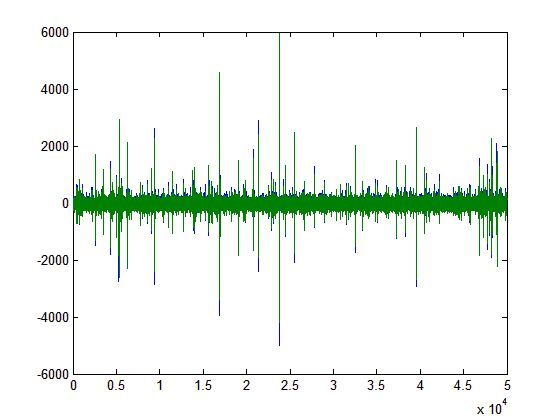}
      \includegraphics [scale=3]{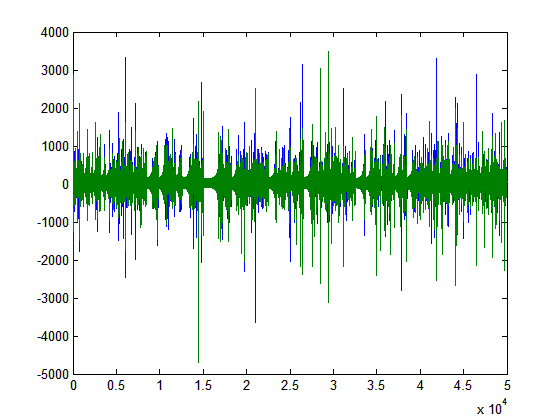}

      \end{tabular}
      }
\caption{Chaotic Solutions for the difference equations Eq. $(1)$ (Top), Eq. $(8)$ (Middle) and Eq. $(9)$ (Bottom).}
      \begin{center}

      \end{center}
      \end{figure}
\noindent
The Lyapunav exponent of the solutions of the three cases as stated in the table are $1.6015$, $1.2414$ and $0.6885$ respectively which declare that the solutions are indeed chaotic.

\begin{openproblem}
Find out the set of all possible initial values $z_0$ and $z_1$ for which the solution are chaotic for given parameters  $\alpha$ and $\beta$.
\end{openproblem}

\section{A Comparative View of Dynamics}
In this section, an attempt has been made to compare the dynamics of the three associated differences equations Eq. $(1)$, Eq $(8)$ and Eq $(9)$. The comparison of the limiting behaviour of the dynamics (sequence of iterates) and the state spaces over $50000$ iterations of the three difference equations are presented in the following two subsections through examples.
\subsection{Comparison of Characteristics of Dynamics}

Here we present a table of comparative dynamics among the three difference equations Eq. $(1)$, Eq $(8)$ and Eq $(9)$. It turns out that all these three difference equations are mostly having different characteristics for same parameters $\alpha$ and $\beta$ and with the initial values $z_0$ and $z_1$. The fixed point of the difference equation Eq. $(8)$ and Eq. $(9)$ is $\alpha+\beta$. We observed that given the fixed parameters $\alpha$ and $\beta$ and the intimal values, any solutions of Eq. $(8)$ which is convergent and converges to the the fixed point $\alpha+\beta$ fetch another solution which is also convergent and converges to the same fixed point of the other difference equation Eq. $(9)$.

\begin{table}[H]
\small
\centering
\begin{tabular}{| m{2.8cm} || m{4cm}|| m{4cm} | | m{4cm}|}
\hline
\centering   \textbf{Initial Values}  &   $z_{n+1}=\frac{\alpha}{z_{n}}+\frac{\beta}{z_{n-1}}$ &   $z_{n+1}=\alpha+\frac{z_n}{z_{n-1}}\beta$ & $z_{n+1}=\frac{z_{n-1}}{z_n}\alpha+\beta$  \\
\hline
\centering $\alpha=(30, 47)$, $\beta=(30, -10)$, $z_0=(9, -41)$, $z_1=(49, -63)$ & Chaotic  & Convergent and converges to $(60, 37)$ & Convergent and converges to $(60, 37)$  \\

\hline
\centering $\alpha=(56, -22)$, $\beta=(-52, -19)$, $z_0=(-81, -74)$, $z_1=(89, 92)$ & Converges to a periodic point $(10.393, -0.14432)$ of period $2$  & Chaotic & Converges to a periodic point $((33.8, 60.46)$ of period $2$   \\
\hline
\centering $\alpha=(4, -81)$, $\beta=(64, 64)$, $z_0=(45, -70)$, $z_1=(32, 4)$ & Converges to a periodic point $(6.9614, -10.414)$ of period $2$  & Divergent & Chaotic \\
\hline
\centering $\alpha=(15, -88)$, $\beta=(-53, -30)$, $z_0=(65, -97)$, $z_1=(-92, -67)$ & Periodic and of Period $23$  & Convergent and converges to $(-38, -118)$ & Convergent and converges to $(-38, -118)$ \\
\hline

\end{tabular}
\caption{Local stability of the equilibriums of Eq. $(9)$}
\label{Table:}
\end{table}

\noindent
The largest Lyapunav exponent of the solutions of the difference equations Eq. (1), Eq. (8) and Eq. (9) of the first three cases have already calculated in the previous section which ensure that the solutions are chaotic.

\subsection{Comparison of State Spaces of Dynamics}

In this section we have adumbrated a list of cases to compare state spaces over $50000$ iterations of the difference equations Eq. $(1)$, Eq. $(8)$ and Eq. $(9)$ in the Table $9$. It turns out that no state space is similar to others.
\begin{table}[H]
\small
\centering
\begin{tabular}{| m{2.8cm} || m{4cm}|| m{4cm} | | m{4cm}|}
\hline
\centering   \textbf{Initial Values}  &   $z_{n+1}=\frac{\alpha}{z_{n}}+\frac{\beta}{z_{n-1}}$ &   $z_{n+1}=\alpha+\frac{z_n}{z_{n-1}}\beta$ & $z_{n+1}=\frac{z_{n-1}}{z_n}\alpha+\beta$  \\
\hline
\centering $\alpha=(9, -73)$, $\beta=(-70, -49)$, $z_0=(52, 110)$, $z_1=(68, 88)$ & Convergent and converges to $(10.591, -5.759)$  & Unbounded  & Convergent and converges to $(79, -122)$  \\

\hline
\centering $\alpha=(100, -55)$, $\beta=(31, 21)$, $z_0=(-82, 160)$, $z_1=(-11, -94)$ & Unbounded  &  Convergent and converges to $(131, -34)$ &  Convergent and converges to $(131, -34)$\\
\hline
\centering $\alpha=(-29, 33)$, $\beta=(-44, -54)$, $z_0=(152, 122)$, $z_1=(87, -191)$ & Fractal & Unbounded & Convergent and converges to $(-73, -33)$ \\
\hline
\centering $\alpha=(58, 56)$, $\beta=(34, -74)$, $z_0=(-8, -59)$, $z_1=(-57, -91)$ & Unbounded  & Unbounded & Convergent and converges to a periodic point $(56.376, -118.56)$ \\
\hline

\centering $\alpha=(98, 1)$, $\beta=(-46, -80)$, $z_0=(-99, 130)$, $z_1=(55, 75)$ & Fractal  & Convergent and converges to a periodic point $(117.82, 14.575)$ of period $6$  & Convergent and converges to a periodic point $(-61.076, -143.49)$ of period $2$\\
\hline
\centering $\alpha=(-64, 0)$, $\beta=(4, 99)$, $z_0=(-89, 184)$, $z_1=(29, -32)$ & Unbounded  & Unbounded & Convergent and converges to $(-60, 99)$ \\
\hline

\centering $\alpha=(80, -87)$, $\beta=(-33, -100)$, $z_0=(87, 64)$, $z_1=(42, 49)$ & Fractal  & Chaotic & Convergent and converges to $(47, -187)$ \\
\hline

\centering $\alpha=(4, 55)$, $\beta=(-76, 25)$, $z_0=(-34, -32)$, $z_1=(64, 6)$ & Unbounded  & Unbounded & Convergent and converges to a periodic point $(19.794, 105.86)$ of period $2$\\
\hline

\end{tabular}
\caption{Characterizations of State space over $50000$ iterations of Eq. $(1)$, Eq. $(8)$ and Eq. $(9)$}
\label{Table:}
\end{table}
\noindent
The states spaces are figured corresponding to each of the above six cases in the Fig. (3). The states space is the set of all $z_i$ for all $i$ ranges from $1$ to $50000$.

\begin{figure}[H]
      \centering
      \resizebox{10cm}{!}
      {
      \begin{tabular}{c c c}
      \includegraphics [scale=2.0]{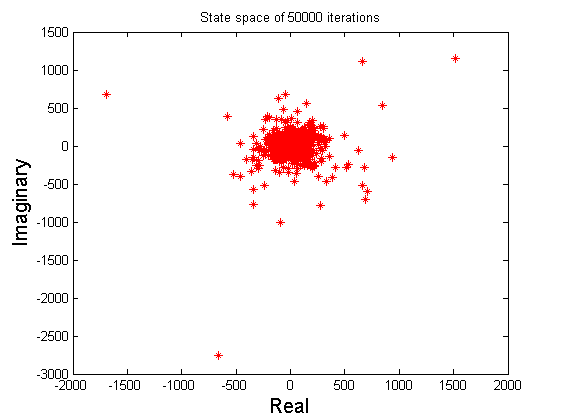}
      \includegraphics [scale=1.45]{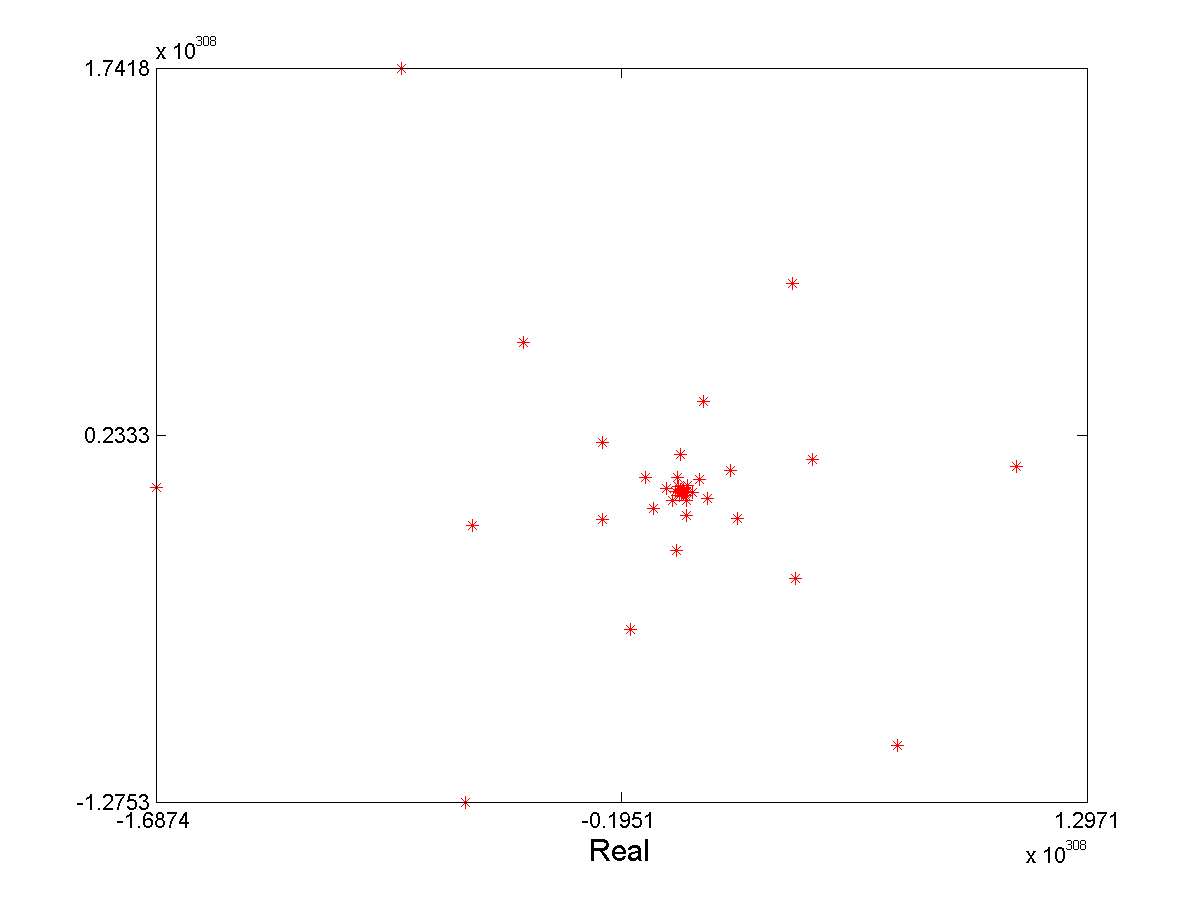}
      \includegraphics [scale=1.45]{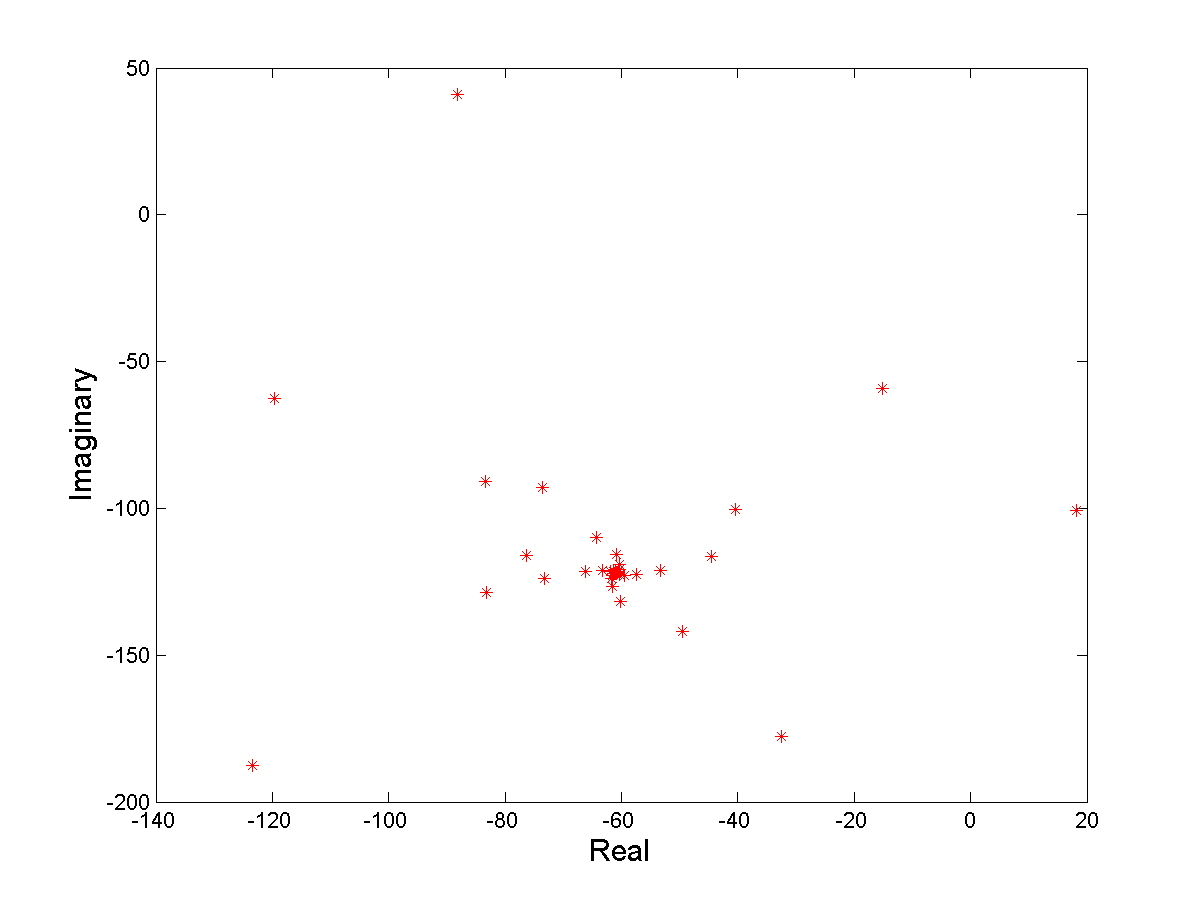}\\
      \includegraphics [scale=2.0]{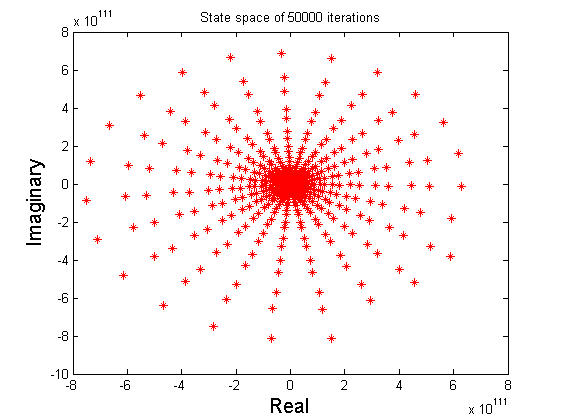}
      \includegraphics [scale=1.45]{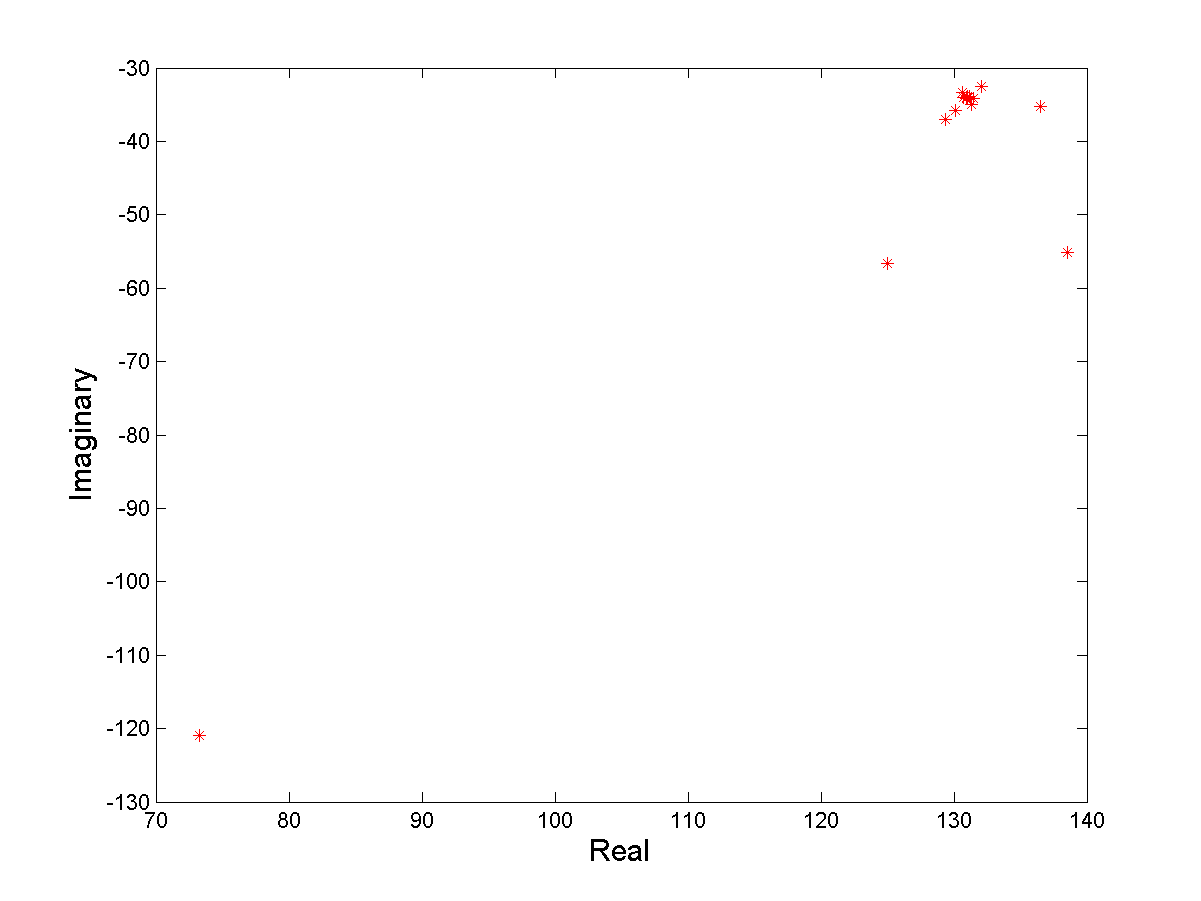}
      \includegraphics [scale=1.45]{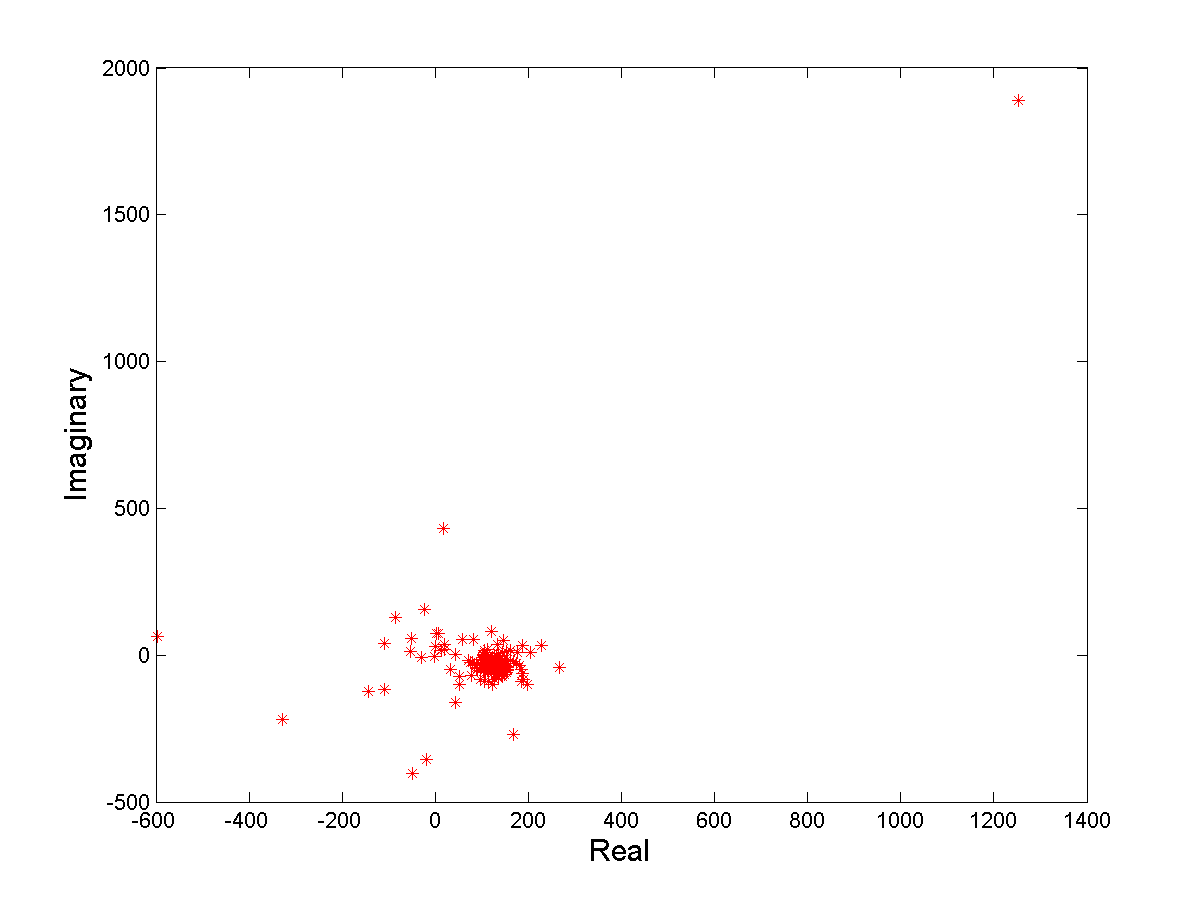}\\
      \includegraphics [scale=2.0]{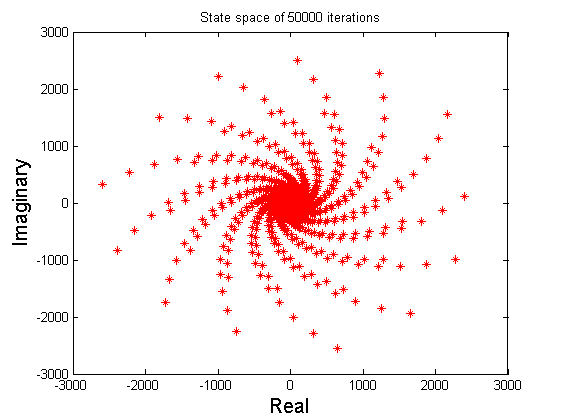}
      \includegraphics [scale=2.0]{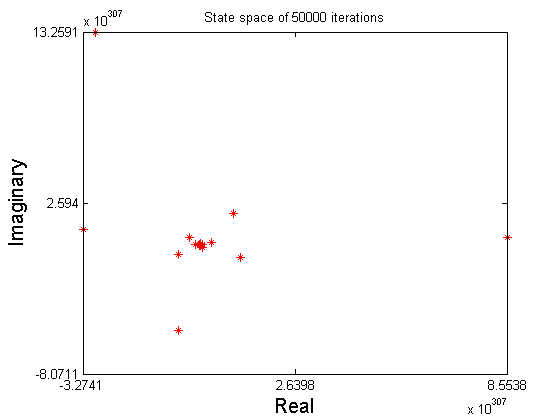}
      \includegraphics [scale=2.0]{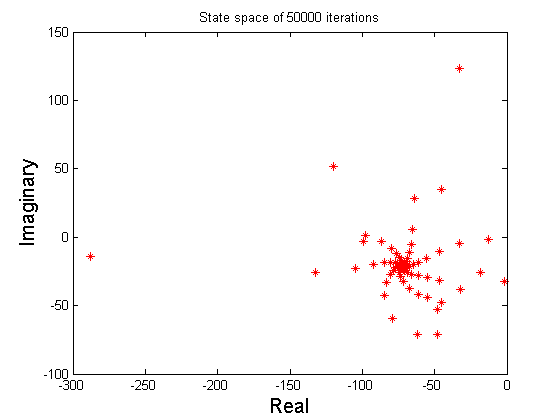}\\
      \includegraphics [scale=2.0]{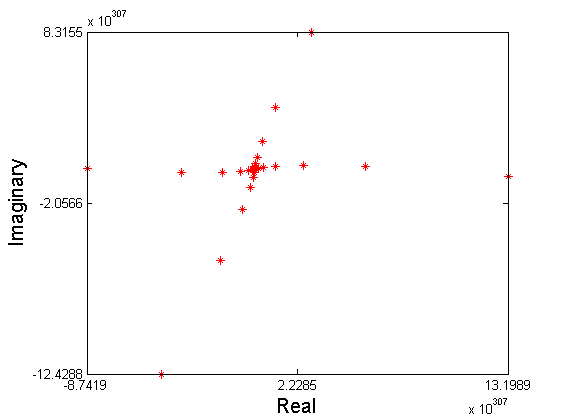}
      \includegraphics [scale=2.0]{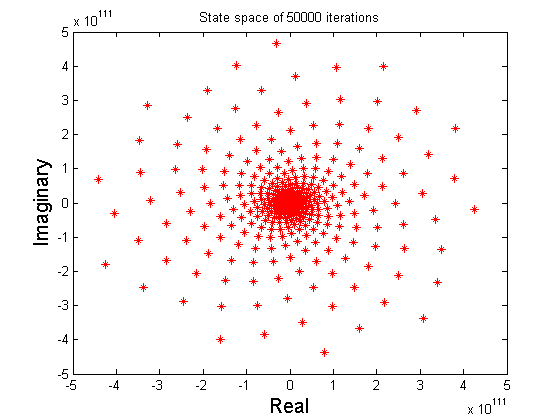}
      \includegraphics [scale=2.0]{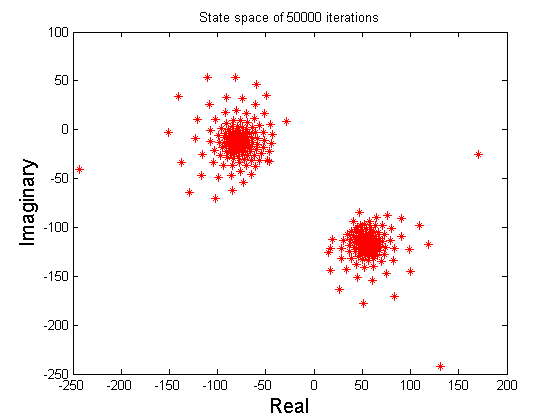}\\
      \includegraphics [scale=2.0]{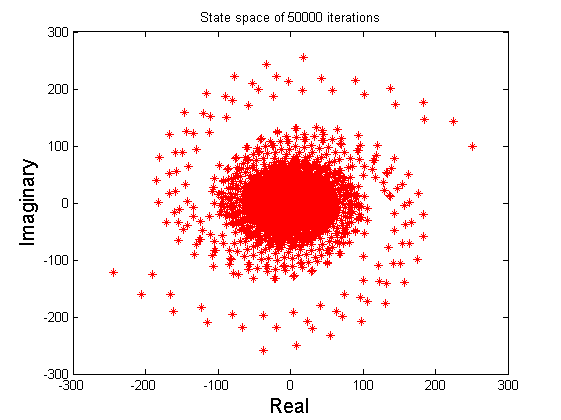}
      \includegraphics [scale=2.0]{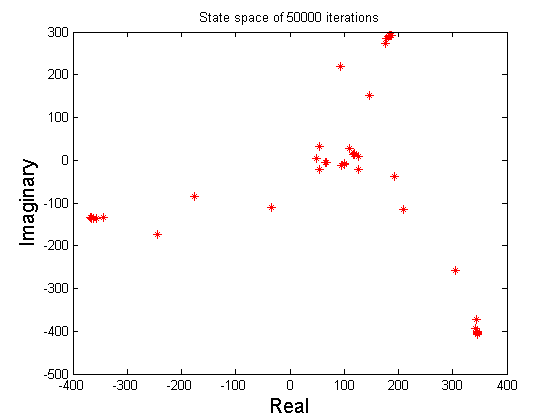}
      \includegraphics [scale=2.0]{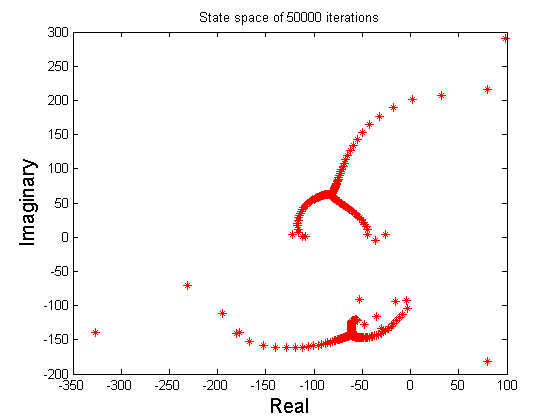}\\
      \includegraphics [scale=2.0]{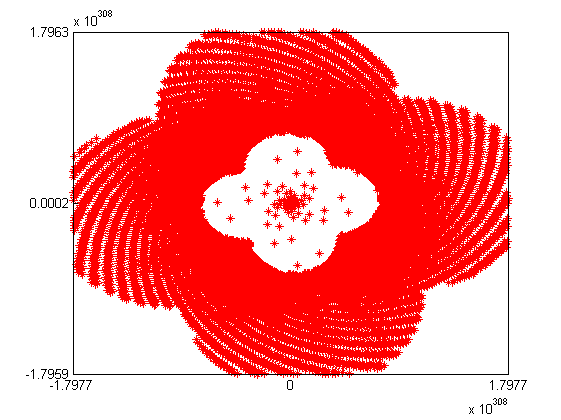}
      \includegraphics [scale=2.0]{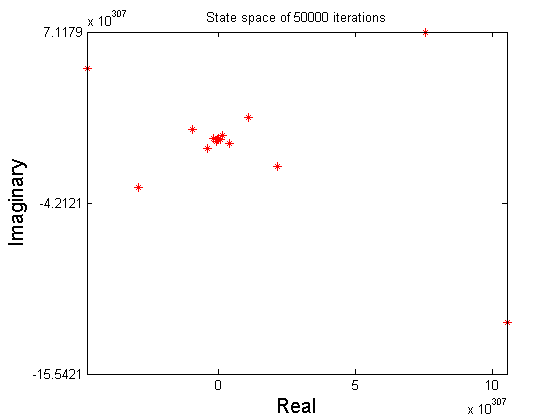}
      \includegraphics [scale=2.0]{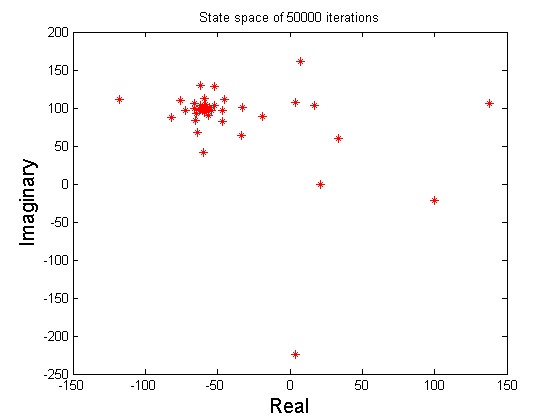}\\
      \includegraphics [scale=1.45]{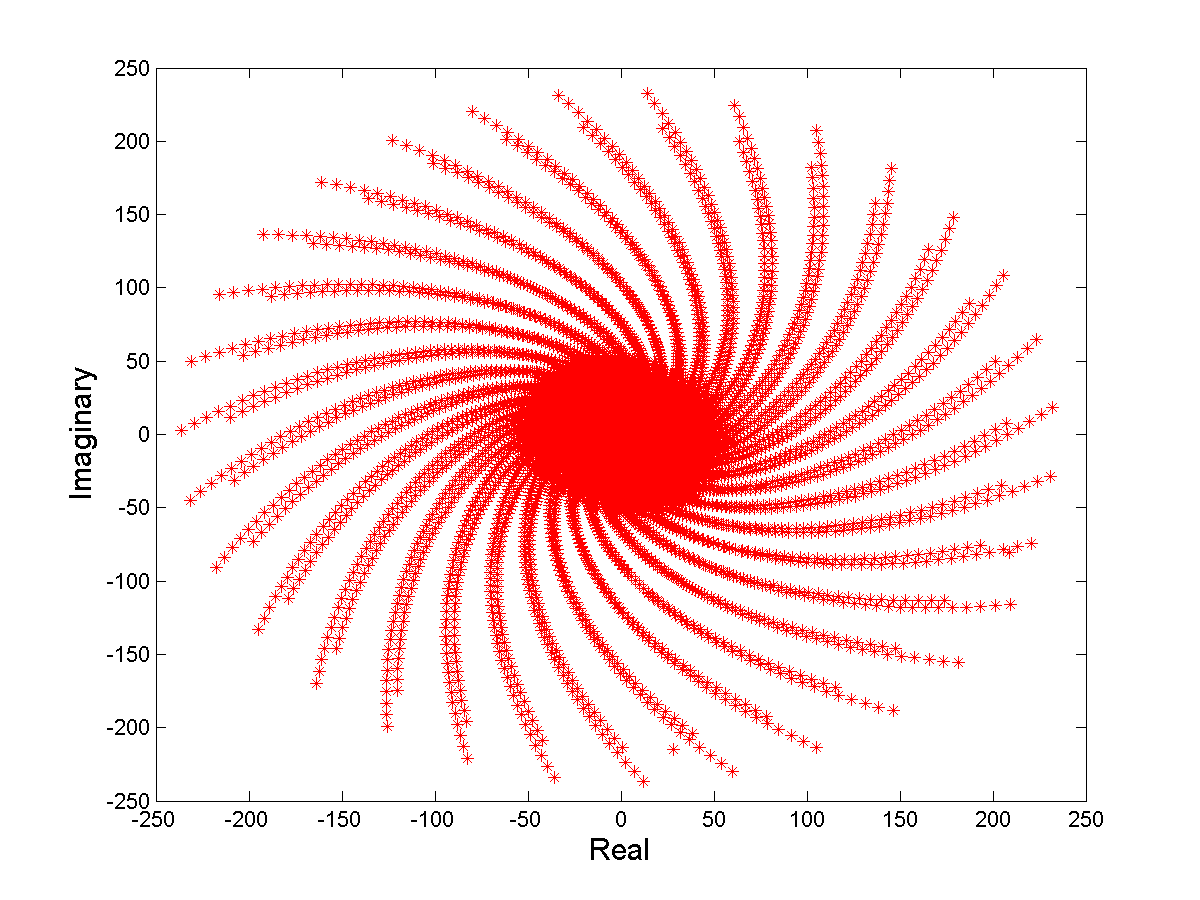}
      \includegraphics [scale=1.45]{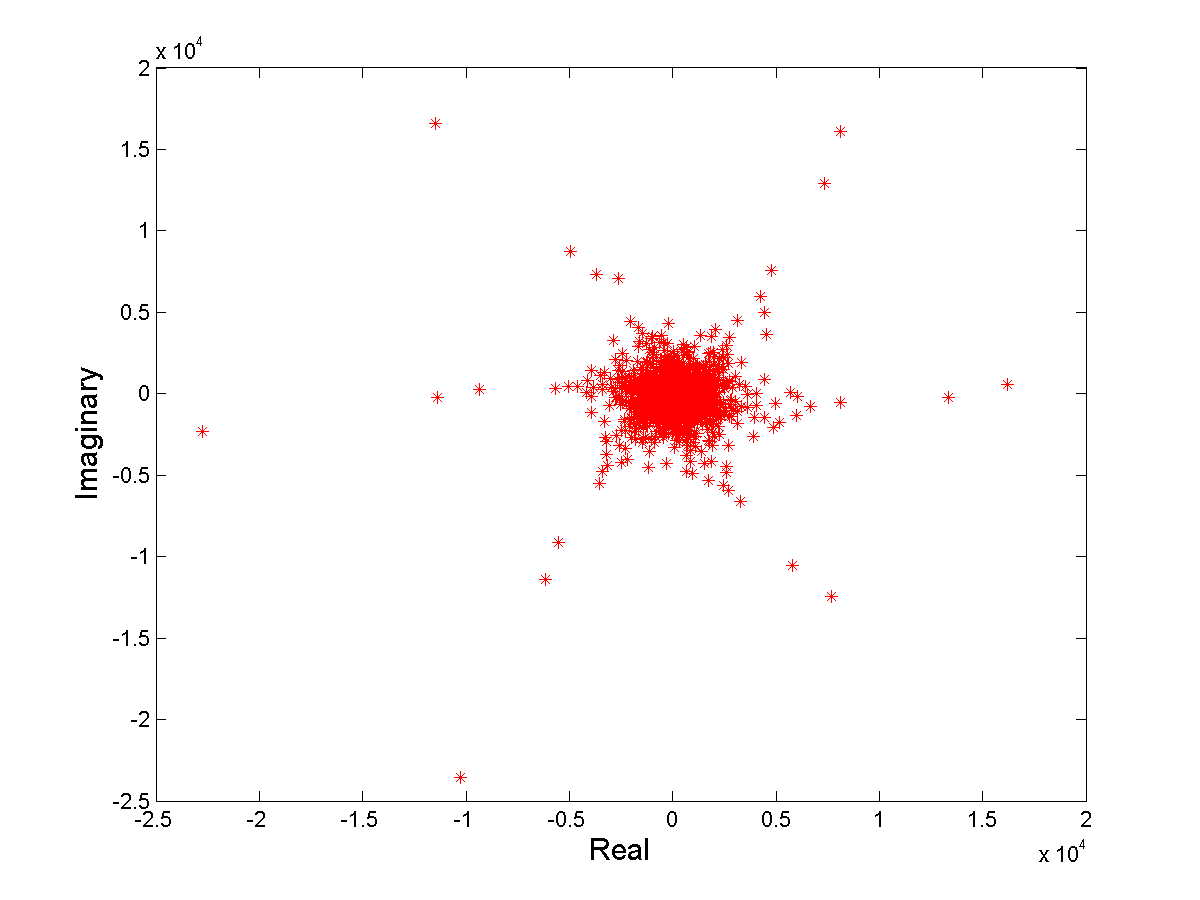}
      \includegraphics [scale=1.45]{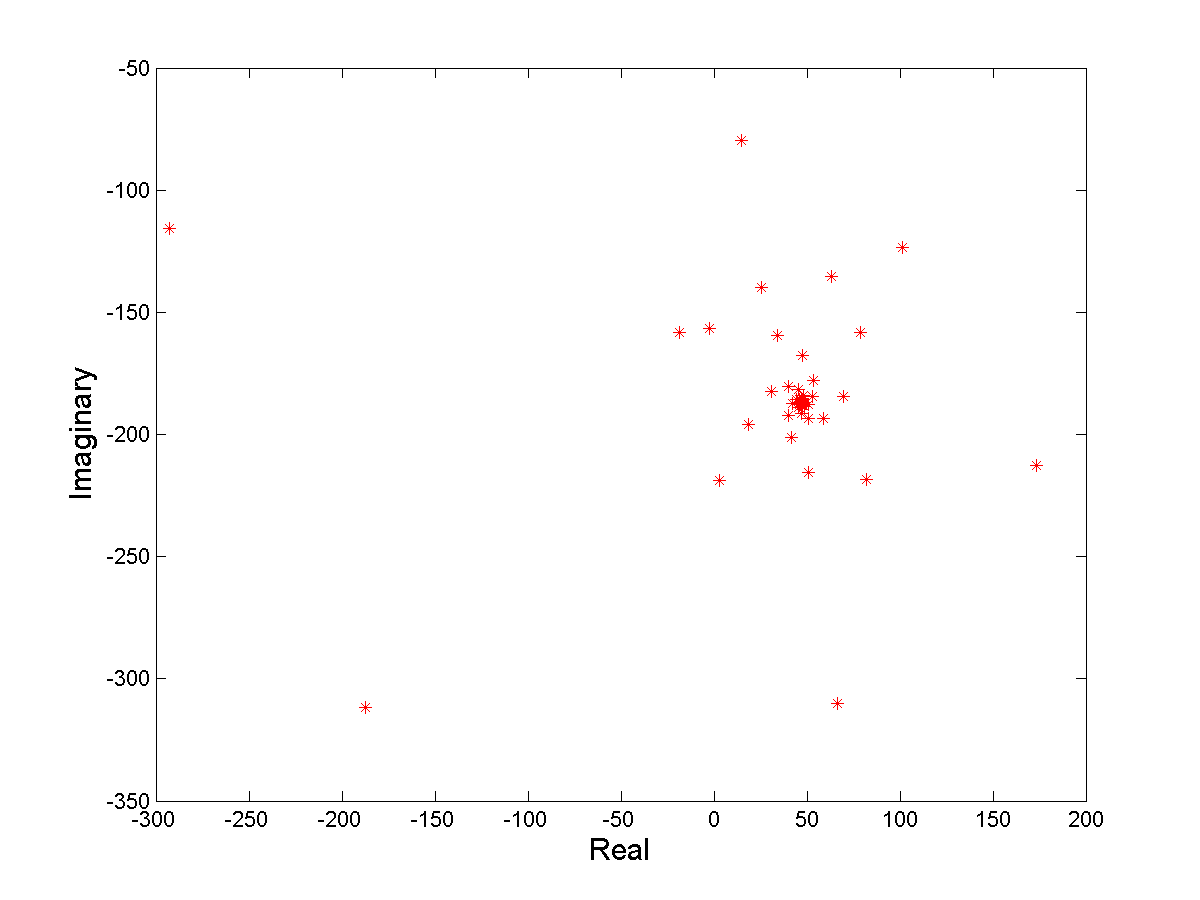}\\
      \includegraphics [scale=1.45]{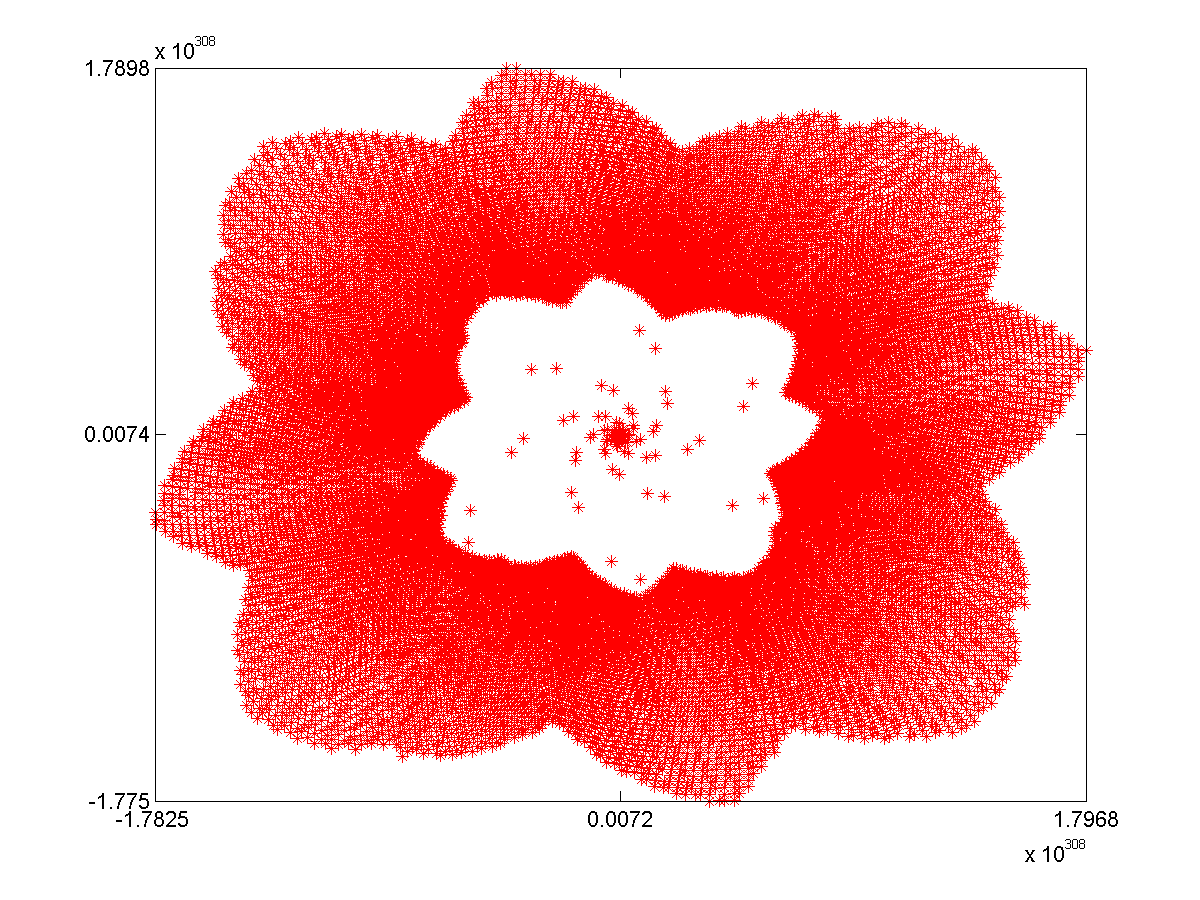}
      \includegraphics [scale=1.45]{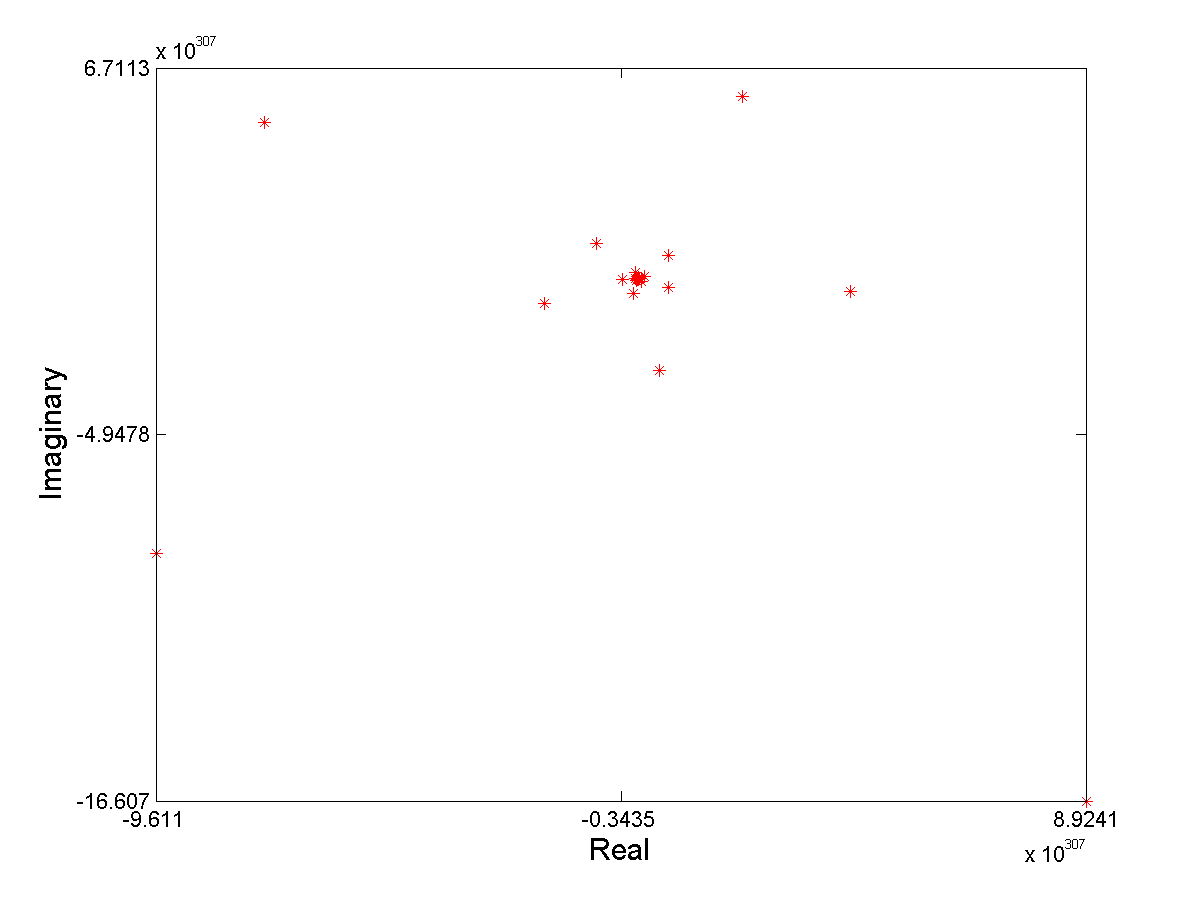}
      \includegraphics [scale=1.45]{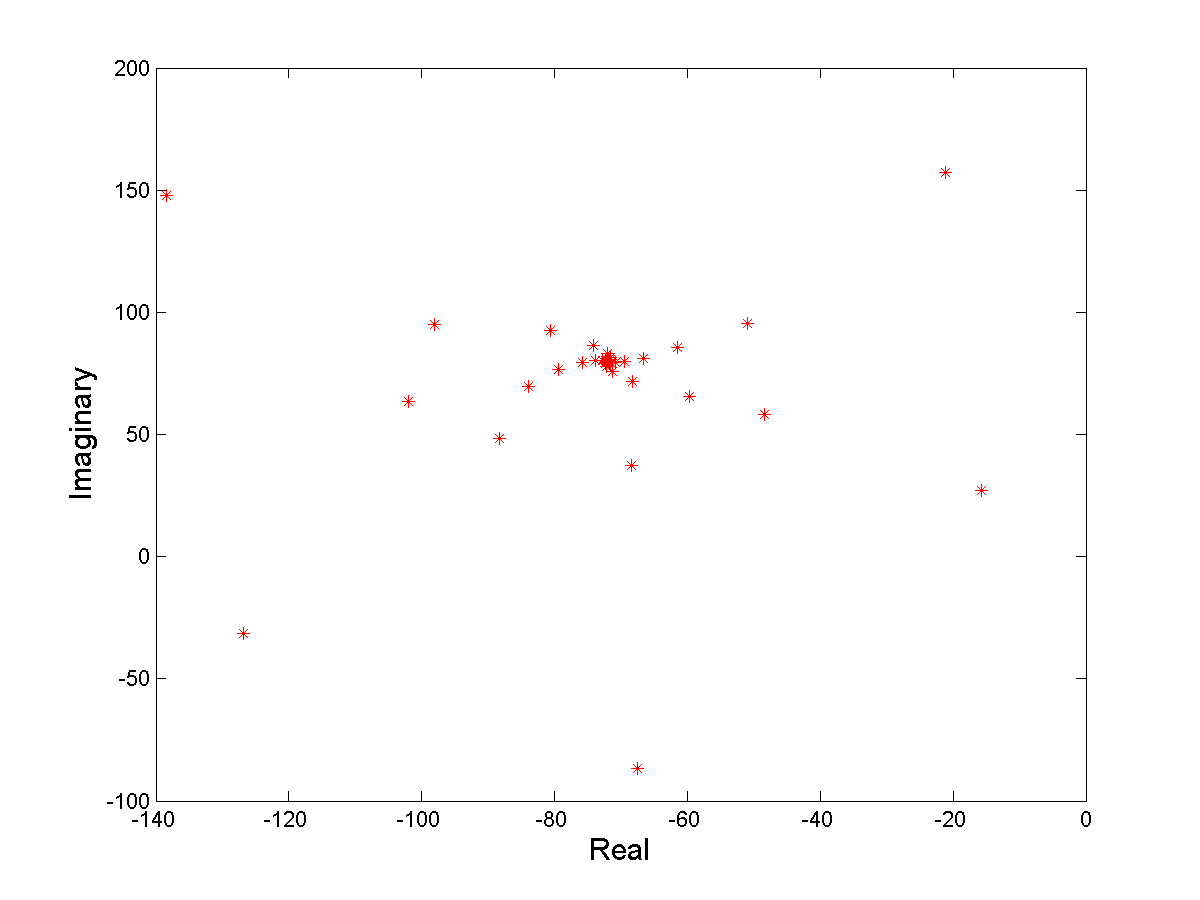}

      \end{tabular}
      }
\caption{States space of $50000$ iterations of the Difference equations Eq. $(1)$, Eq. $(8)$ and Eq. $(9)$. Each row of the figures denotes each different case as states in the Table $(9)$ respectively.}
      \begin{center}

      \end{center}
      \end{figure}
\noindent
In the case of $3$rd, $5$th and $7$th cases of the Table $(9)$, the solutions of the difference equation Eq. (1) are turned out to be fractal. The complex sequence plot of the different number of iterations are given in the in Fig. $(4)$ for the $3$rd and $5$th cases.

\begin{figure}[H]
      \centering

      \resizebox{16cm}{!}
      {
      \begin{tabular}{c c c}
      \includegraphics [scale=2.0]{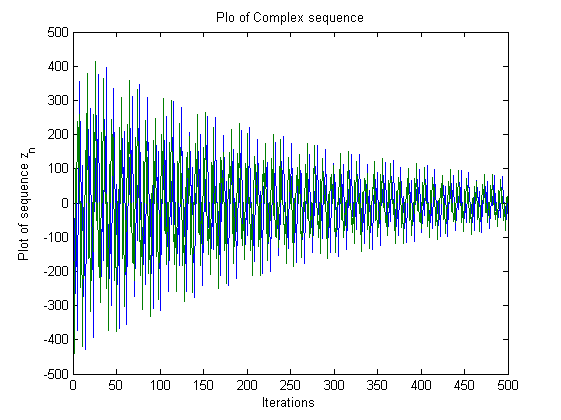}
      \includegraphics [scale=2.0]{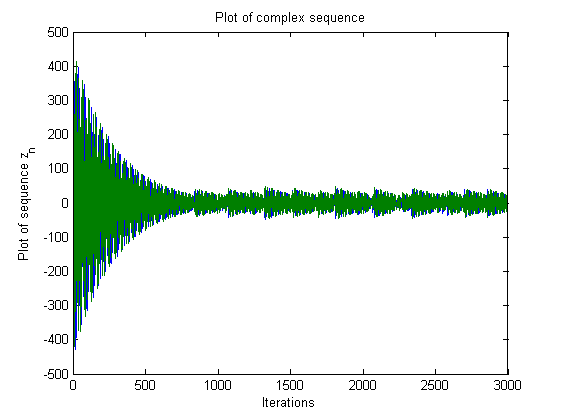}
      \includegraphics [scale=2.0]{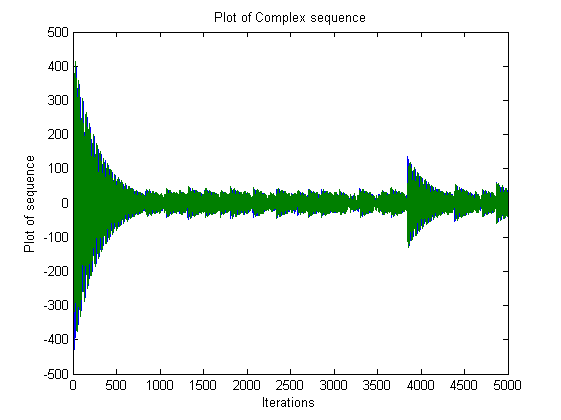}\\
      \includegraphics [scale=2.0]{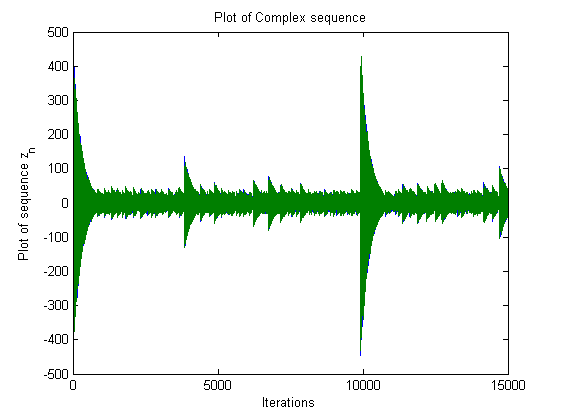}
      \includegraphics [scale=2.0]{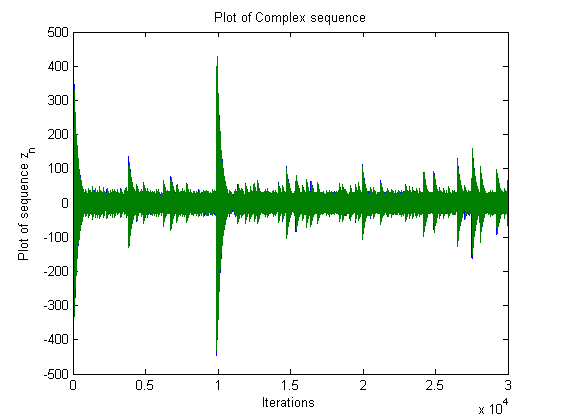}
      \includegraphics [scale=2.0]{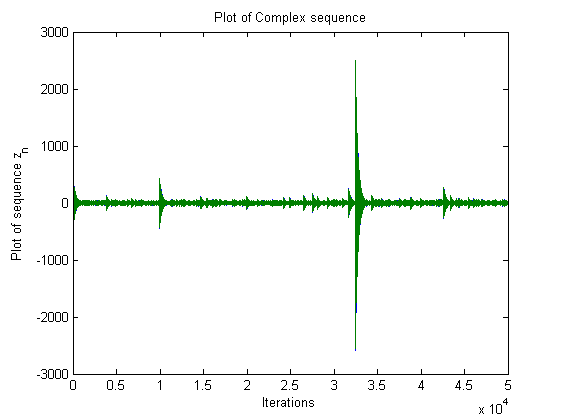}\\
      \includegraphics [scale=2.0]{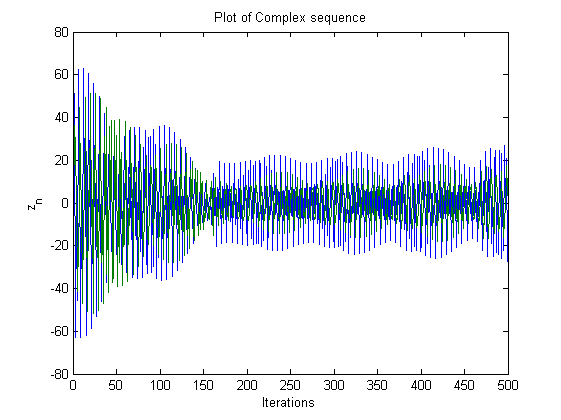}
      \includegraphics [scale=2.0]{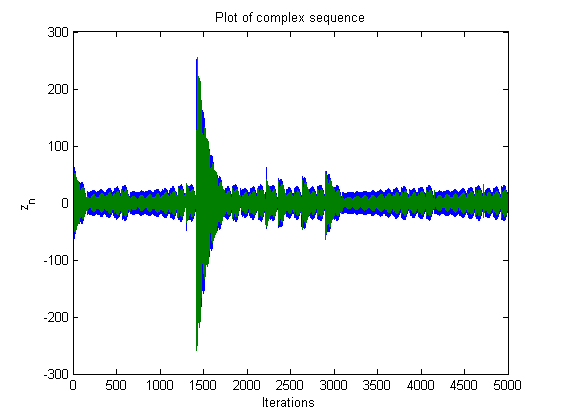}
      \includegraphics [scale=2.0]{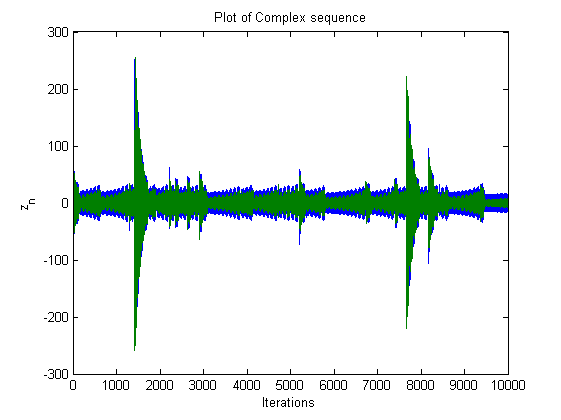}\\
      \includegraphics [scale=2.0]{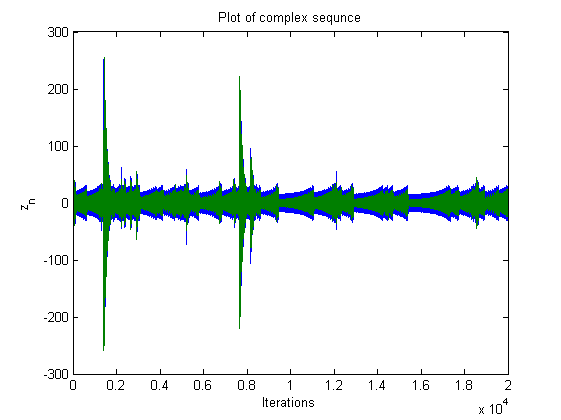}
      \includegraphics [scale=2.0]{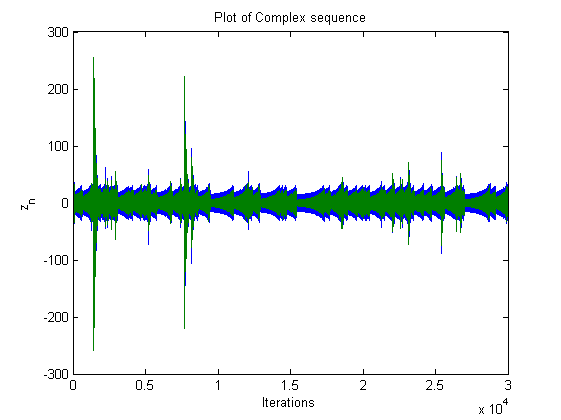}
      \includegraphics [scale=2.0]{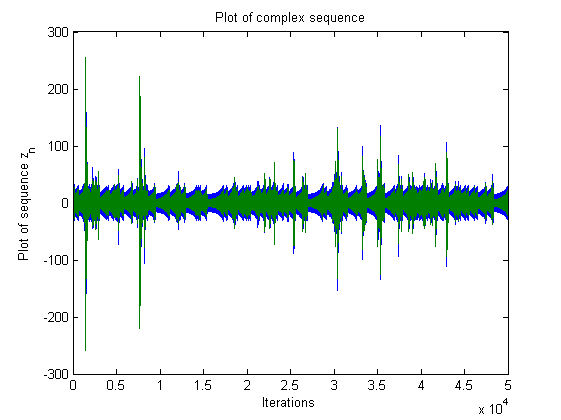}\\

      \end{tabular}
      }

      \begin{center}
\caption{Plot of Complex sequences which turns out to be a fractal over different number of iterations; First two rows and next consecutive two rows of figures for $3$rd and $5$th cases respectively}
      \end{center}
      \end{figure}
\noindent
The fractal dimension of the fractal solution for the $3$rd, $5$th and $7$th cases are $1.82779$, $1.89333$ and $1.9127$ respectively. This fractal dimension have been calculated through the software \texttt{$Benoit^{TM}$}.
\section{Conclusion}

\label{section:conclusion}

Nonlinear difference equations are an established dynamic area of research in both real and complex domain, but the study of rational difference equations is still in its infancy. The first order Riccati difference equation, for which the solution could be obtained explicitly, has been studied both in the real and complex domains, see \cite{G-K-S} and \cite{K-L}. Second and third order rational difference equations with positive real parameters and initial conditions, have also been studied extensively, see \cite{Ca-L}. Also see \cite{Ca-Ch-L-Q-1} and \cite{Ca-Ch-L-Q-2}.

\addvspace{\bigskipamount}

When we consider the parameters and initial conditions to be complex, there is hardly anything known about rational difference equations. In the present manuscript, we took three closely related difference equations and did lot of computational study in understanding their dynamics characteristics. It is turned out as expected the richness and complexity of the dynamics in terms on chaotic, fractal, unbounded solutions where in the case of real parameters and real initial values nothing were found for the same difference equations. Our main purpose behind this manuscript is to probe and analyze how the dynamics changes for rational equations when we shift from positive real domain to the complex domain.

\addvspace{\bigskipamount}

We have carried out some local stability and global periodicity analysis. We have observed computationally the different characteristics of boundedness, unboundedness, convergence, convergence to periodic solutions, fractal and chaos. Each one of these characteristics of the dynamics is important in their own right and we have posed them as challenging open problems which require further theoretical investigations. Further, certain peculiar behavior of the dynamics in the complex domain has been brought forward which is in stark contrast to the positive real scenario.

\addvspace{\bigskipamount}

The work carried out is generic in nature and can be helpful towards the study of other rational difference equations of second and higher orders. It is our believe, that this work would initiate the study and understanding of rational difference equations in the complex plane.

\section*{Acknowledgement}
The author SSH thanks \emph{Karthik, Kasi, Arita} of ICTS, TIFR and \emph{Dr. Esha Chatterjee} of IISc for discussions and suggestions in many stages of the present research.


\end{document}